\documentclass[a4paper,10pt]{article}
\usepackage{graphicx}
\usepackage{float}
\usepackage{amsthm}
\usepackage{amsfonts}
\usepackage[latin1]{inputenc}
\usepackage{indentfirst}
\usepackage{amsmath}
\usepackage{amssymb}
\usepackage{color}

\newtheorem{theo}{Theorem}[section]

\newtheorem{lemma}[theo]{Lemma}
\newtheorem{cor}[theo]{Corollary}
\newtheorem{prop}[theo]{Proposition}

\newtheorem{ann}{Assumption}

\newtheorem*{Rk}{Remark}

\numberwithin{equation}{section}

\def\EE{\mathcal{EE}}

\def\Z{\mathbb{Z}}
\def\N{\mathbb{N}}

\def\P{\mathbb{P}}

\def\E{\mathbb{E}}

\def\EE{\mathcal{E}}

\def\R{\mathbb{R}}

\def\t{\textrm}

\newcommand{\be} {\begin{equation}}
\newcommand{\ee} {\end{equation}}
\newcommand{\bea} {\begin{eqnarray}}
\newcommand{\eea} {\end{eqnarray}}
\newcommand{\Bea} {\begin{eqnarray*}}
\newcommand{\Eea} {\end{eqnarray*}} 

\begin{document}
\title{Small positive values for supercritical branching processes in random environment}
\author{Vincent Bansaye and Christian B\"oinghoff}
\maketitle \vspace{1.5cm}

\begin{abstract} Branching Processes in Random Environment (BPREs) $(Z_n:n\geq0)$  are the generalization of Galton-Watson processes where 
in each generation the reproduction law is picked randomly in an i.i.d. manner. In the supercritical case, the process survives 
with positive probability and then almost surely grows geometrically. This paper focuses on rare events when the process takes positive but small values for large times. \\
We describe the asymptotic behavior of $\P(1 \leq Z_n \leq k|Z_0=i)$, $k,i\in\mathbb{N}$ as $n\rightarrow \infty$. More precisely, 
we characterize the exponential decrease of $\P(Z_n=k|Z_0=i)$ using
a spine  representation due to Geiger. We then provide some bounds for this rate of decrease. \\
If  the reproduction laws are linear fractional, this rate becomes more explicit and two regimes appear. Moreover, we show that these regimes  affect the asymptotic behavior of 
 the most recent common ancestor, when the population is conditioned to be small but positive for large times.
\end{abstract}


\maketitle

\section{Introduction}
A branching process in random environment (BPRE) is a discrete time and discrete size population model going back to \cite{smith69,athreya71}. In each generation, an
offspring distribution is picked at random, independently from one generation to the other. We can think of a population of plants having a one-year life-cycle.
In each year, the outer conditions vary in a random fashion. Given these conditions, all individuals reproduce independently according to the same mechanism. Thus, it satisfies both the Markov and branching properties. \\
Recently, the problems of rare events and large deviations have attracted attention  \cite{kozlov06, bansaye08,BK09,kozlov10,HuangLiu, BB10}. 
However, the problem of small positive values has not been treated except in the easiest case which assumes non-extinction, 
i.e. $\P(Z_1\geq 1|Z_0=1)\geq 1$ (see \cite{bansaye08}).
For Galton-Watson processes, the explicit  equivalent of this probability is well-known (see e.g. \cite{AN}[Chapter I, Section 11, Theorem 3]). In particular, denoting by $f$ the probability generating function of the reproduction law, we have for  $k$ large enough,

\be
\label{asGW}
\varrho:=\lim_{n\rightarrow \infty} \tfrac{1}{n} \log \P(1\leq Z_n\leq k \vert Z_0=1)=f'(p_e), \  \t{where} \quad  p_e=\P(\exists n \in \N : Z_n=0 \ \vert \ Z_0=1).
\ee
Moreover, the rate of decrease remains equal to $\varrho$ if $k_n$ decreases sub-exponentially. 
It means that as soon as $k_n/\exp(\delta n) \rightarrow 0$ as $n\rightarrow \infty$ for every $\delta >0$, 
then $\lim_{n\rightarrow \infty} \frac{1}{n} \log \P(1\leq Z_n\leq k_n \ \vert \ Z_0=1)=\varrho$.
In this paper, we focus on the existence and characterization of $\varrho$ in the random environment framework. 
It is organized as follows.\\
 First, we give the classical notations and properties of BPRE. In the next section, we state our results. We prove the existence of $\varrho$ and a characterization of its value via a spine construction, give a lower bound and an upper bound which have natural interpretations. 
Finally, we specify our results in the linear fractional case, where two regimes appear, 
which are also visible in the time of the most recent common ancestor (MRCA).\\
In the rest of the paper, the proofs of these results are presented. Section \ref{section4} deals with a tree construction due to Geiger, which is used in Section \ref{proofs} 
to  characterize $\rho$. In Section \ref{rho0}, we prove that $\rho>0$ under suitable assumptions.  
In Section \ref{section6}, we prove a lower bound for $\varrho$ in terms of the rate function of the associated random walk.  
Finally, 
in Sections  \ref{linearfrac} and \ref{secmrca} the statements for the linear fractional case are proved using the general results obtained before, whereas in Section \ref{examples}, 
we present some details on two examples. \\

For the formal definition of a branching process $Z$ in random environment, let $Q$ be a random variable taking values in $\Delta$, the space of all probability measures
 on $\N_0=\{0,1,2,\ldots\}$. We denote by $$m_q:=\sum_{k\geq 0} k\ q(\{k\})$$
 the mean number of offsprings of $q\in \Delta$. For simplicity of notation, we will shorten 
$q(\{\cdot\})$ to $q(\cdot)$ throughout this paper.
 An infinite sequence
$\mathcal{E}=(Q_1,Q_2,\ldots)$ of independent, identically distributed (i.i.d.) copies of $Q$ is called a random environment. Then the integer valued process
 $(Z_n : n\geq 0)$ is called a branching process in 
the random environment $\mathcal{E}$ if $Z_0$ is independent of $\mathcal{E}$ and it satisfies 

\begin{equation}  \label{transition} 
    \mathcal{L} \big(Z_{n} \; \big| \; \, \mathcal{E}, \  Z_0,\dots, Z_{n-1}
    \big) \ = \ Q_{n}^{*Z_{n-1}} \qquad \text{a.s.}
\end{equation} 
for every $n\geq 1$, where $q^{*z}$  is the $z$-fold convolution of the measure $q$. We introduce the probability generating function (p.g.f)
of $Q_n$, which is denoted by $f_n$ and defined by 
$$f_n(s):=\sum_{k=0}^\infty s^k Q_n(k), \qquad (s\in [0,1]).$$ 
In the whole paper, we denote  indifferently  the associated random environment by $\mathcal{E}=(f_1, f_2, ...)$ and $\mathcal{E}=(Q_1, Q_2, ...)$. The characterization (\ref{transition}) of the law of $Z$  becomes
$$\E\big[s^{Z_{n}}\vert \mathcal{E}, \  Z_0,\dots,Z_{n-1}\big]=f_n(s)^{Z_{n-1}} \qquad \text{a.s.}\qquad (0\leq s\leq 1, n \geq 1).$$

$ \qquad$ Many  properties of $Z$ are mainly determined by the \textbf{random walk associated with the environment} $(S_n \ : \ n\in\N_0)$ which is defined by
$$S_0=0, \qquad S_{n}-S_{n-1}=X_n \quad (n\geq 1),$$ 
where $$X_n:=  \log m_{Q_n}=\log f_n'(1)$$
are i.i.d. copies of the logarithm of the mean number of offsprings  
$X:=\log (m_Q)=\log(f'(1))$. \\
Thus, one can check easily that
\begin{eqnarray}
\E[Z_n| Q_1,\ldots,Q_n, Z_0=1] &=& e^{S_n} \quad \mbox{a.s.}  \label{ew1}
\end{eqnarray}
There is a well-known classification of BPRE \cite{athreya71}, which we recall here in the case $\mathbb E[\vert X \vert] <\infty$. 
In the subcritical case ($\mathbb{E}[X]<0$), the population becomes extinct at an exponential rate in almost every environment. Also in the critical case ($\mathbb{E}[X]=0$), the process 
becomes extinct a.s. if  we exclude the degenerated case when $\P(Z_1=1|Z_0=1)=1$.  
In the supercritical case ($\mathbb{E}[X]>0$), the process survives with positive probability under quite general assumptions on the offspring distributions 
(see \cite{smith69}). Then   $\E[Z_1\log^+(Z_1) / f'(1)]<\infty$ ensures that the martingale $ e^{-S_n}Z_n$ has a positive finite limit 
on the non-extinction event:
$$\lim_{n\rightarrow \infty} e^{-S_n}Z_n =W\ \text{a.s.}, \qquad \P(W>0)=\P(\forall n  \in \N : Z_n>0|Z_0=1)>0.$$ 
 
The problem of small positive values is linked to the left tail of $W$ and the existence of harmonic moments. In the Galton-Watson case, we refer to \cite{athreya, Rouault, FlWa}. 
For BPRE, Hambly \cite{Hambly} gives the  tail of $W$ in $0$, whereas  Huang \& Liu \cite{HuangLiu, HuangLiu2} have obtained other various results in this direction.  \\

\section{Probability of staying bounded without extinction}\label{section2}

Given the initial population size $k$, the associated probability is denoted by $\mathbb{P}_k(\cdot):=\mathbb{P}(\cdot|Z_0=k)$. For convenience, 
we write  $\mathbb{P}(\cdot)$ when  the size of the initial population is taken equal to $1$ or does not matter.
Let $f_{i,n}$ be the probability generating function of $Z_n$ started in generation $i\leq n$ :
$$f_{i,n}:=f_{i+1} \circ f_{i+2} \circ \cdots \circ f_{n}, \quad f_{n,n}=Id \qquad \t{a.s.} $$

We will now specify the asymptotic behavior of $\P_i(Z_n=j)$ for $i,j\geq 1$, which  may depend both on $i$ and $j$. One can first observe
that  some integers $j$ cannot be reached by $Z$ starting 
from $i$ owing to the support of the offspring distribution. \\
The first result below introduces the rate of decrease $\varrho$ of $\P_i(Z_n=j)$ for $i,j\geq 1$ and gives a trajectorial  interpretation of the associated rare event $\{Z_n=j\}$. 
The second one provides general conditions to ensure that $\varrho>0$. It also gives an  upper bound  of  $\varrho$, which may be reached, 
in terms of the rate function of the random walk $S$. This bound corresponds to the environmental stochasticity, which means that the rare event $\{Z_n=j\}$ 
is explained by rare environments. 
The next result yields an explicit expression of the rate $\varrho$ in the case of  linear fractional offspring distributions, 
where two supercritical regimes appear. The last corollary considers the most recent common ancestor, where a third regime appears which is 
located at the borderline of the phase transition. \\

Let us define 
$$\mathcal{I}:=\big\{ j \geq 1 \ :\ \P(Q(j)>0, Q(0)>0)>0\big\}$$
and introduce the set $Cl(\{z\})$ of integers that can be reached from $z\in\mathcal{I}$, i.e.
$$ Cl(\{z\}):=\big\{ k\geq 1 : \exists n \geq 0  \text{ with } \ \P_z(Z_n=k)>0\big\}. $$
Analogously, we define $Cl(\mathcal{I})$ as the set of integers which can be reached from $\mathcal{I}$ by  the process $Z$. More precisely,  
$$ Cl(\mathcal{I}):=\big\{ k\geq 1 : \exists n \geq 0  \text{ and } j\in\mathcal{I} \text{ with } \ \P_j(Z_n=k)>0\big\}. $$ 
We observe that $\mathcal{I} \subset Cl(\mathcal{I})$ and  if $\P(Q(0)+Q(1)<1)>0$ and $\P(Q(0)>0,  \ Q(1)>0)>0$, then $Cl(\mathcal{I})=\mathbb{N}$. \medskip\\
We are interested in the event $\{Z_n=j\}$ for large $n$.  
First, we  recall that the case $\mathbb{P}(Z_1=0)=0$ is easier and the rate of decrease of the probability is known  \cite{bansaye08}.
Indeed, then  $Z$ is nondecreasing and for 
$k\geq j \in \N$ such that $\P_k(Z_l=j)>0$ for some $l\geq 0$, we have
$$\lim_{n\rightarrow \infty} \tfrac{1}{n} \log \P_k( Z_n=j)=-k\varrho,  \qquad \text{with } \varrho=-\log \P_1(Z_1=1).$$
We note that in the case $\P_1(Z_1=1)=\P(Z_1=0)=0$, the branching process grows exponentially in almost every environment and 
the probability on the left-hand side is zero if  $n$ is large enough. \\
Thus, let us now focus on the supercritical case with possible extinction, which ensures that  $\mathcal{I}$ is not empty. The expression 
of $\rho$ in the next theorem will be used to get the other forthcoming results.
\begin{theo}\label{theo_rho1} We assume that $\E[X]>0$ and $\P(Z_1=0)>0$. Then 
the following limits exist and coincide for all $k,j\in Cl(\mathcal{I})$,
$$\varrho:=-\lim_{n\rightarrow \infty} \tfrac{1}{n} \log \P_k( Z_n=j)=
-\lim_{n\rightarrow\infty} \tfrac{1}{n} \log \E\big[Q_{n}(z_0)f_{0,n}(0)^{z_0-1} \Pi_{i=1}^{n-1} f_i'\big(f_{i,n}(0)\big)\big]$$
where  $z_0$ is the smallest element of $\mathcal{I}$.  The common limit $\varrho$ belongs to $[0,\infty)$.
\end{theo}
The proof  is given in Section \ref{proofs} and  results from Lemmas  \ref{rho_1} and \ref{lemchar}. 

The right-hand side expression of $\varrho$  
correspond to the event $\{Z_n=j\}$ explained by a ``spine structure''. More precisely,  one individual survives until generation $n$ and gives birth  to the $j$ survivors in the very last generations, 
whereas the other subtrees become extinct (see forthcoming Lemma \ref{le_subtree2} for details). 
However, we are seeing in the linear fractional case (Corollary \ref{lf1}) that a multi-spine structure can also explain $\{Z_n=j\}$ in some regime. 
Thus the optimal strategy is nontrivial and will here only be discussed in the linear fractional case.\\ 
 The proof of Theorem \ref{theo_rho1} is easy if we consider the limit of $\tfrac{1}{n} \log \mathbb{P}_1(Z_n=1)$ as $n\rightarrow\infty$. 
In this case, a direct calculation of the first derivative
of $f_{0,n}$ yields the claim. However, the proof for the general case is more involved. Here, we use probabilistic arguments, which rely  on a 
spine decomposition of the conditioned branching tree via Geiger construction. \\

We  also note that we need to focus on $ i,j\in Cl(\mathcal{I})$. Indeed, 
 $\lim_{n\rightarrow\infty} \tfrac 1n\log\mathbb{P}_1(Z_n=i)$ 
and \\
$\lim_{n\rightarrow\infty} \tfrac 1n\log\mathbb{P}_1(Z_n=j)$
may both exist and be finite for $i\neq j$, but have different values.  To see that, one can consider two environments
$q_1$ and $q_2$ such that 
$$\mathbb{P}(Q_1=q_1)=1-\mathbb{P}(Q_1=q_2)>0; \quad  q_1(1)=1; \quad 
q_2(0) + q_2(2) =1.$$
Moreover, the case 
$-\infty<\limsup_{n\rightarrow\infty} \tfrac 1n\log\mathbb{P}_k(Z_n=k)<\liminf_{n\rightarrow\infty} \tfrac 1n\log\mathbb{P}_1(Z_n=k)<0$ with $k>1$
is also possible. These results are  developed in the two examples given in Section \ref{examples} at the end of this paper. \medskip\\
In the Galton-Watson case, $f$ is constant and for every $i\geq 0$, $f_i=f$ a.s. Then $f_{i,n}(0) \rightarrow p_e$ as $n\rightarrow \infty$ and we recover  the classical result (\ref{asGW}).
 \\

The results and  remarks above  could lead to the conjecture $\varrho=-\log \E[f'(p(f))]$, where $p(f)=\inf \{ s \in [0,1] : f(s)=s\}$. Roughly speaking, 
it would correspond to integrate the value obtained in the Galton-Watson case with respect to the environment. 
The two following results  show that this is not true in general. \\

To prove    that the probability of staying small but alive decays exponentially (i.e.  $\rho>0$) requires some assumptions. 
To avoid too much technicalities, we are assuming 
\begin{ann}\label{technical2}
There exists  $\gamma>0$ such that  $Q(0)<1-\gamma$ a.s. and $\mathbb{E}[|X|]<\infty$.
\end{ann}

Similarly,  to  give an upper bound of $\varrho$ in terms of the rate function of the random walk $S$, 
we require  the following Assumption. The non-lattice condition is only required for a functional limit result 
which is taken from \cite{abkv10}, whereas the truncated moment assumption is classically used for lower bounds of the survival probability of BPREs.

\begin{ann}\label{axi} We assume that  $S$ is non-lattice, i.e. for every $r>0$, $\P( X \in r\Z)<1$.\\
Moreover,  there exist  $\varepsilon >0$ and  $a\in\mathbb{N}$ such that for every $x>0$,
$$\mathbb{E}\big[(\log^+ \xi_Q(a))^{2+\varepsilon}|X>-x\big]<\infty,$$
where $\log^+ x:=\log(\max(x,1))$ and $\xi_q(a)$ is the  truncated standardized second moment 
 \begin{align*}
 \xi_q(a):=\sum_{y=a}^\infty y^2 q(y)/m_q^2, \quad a\in\mathbb{N}, q \in \Delta.
\end{align*}
\end{ann}

\begin{prop}\label{prop2} We assume that there exists  $s>0$ such that $\mathbb{E}[e^{-sX}]<\infty$.

(i) If   Assumption \ref{technical2} is fulfilled, then $\rho>0$.

(ii)  If $\P(X\geq 0)=1$ or  Assumption \ref{axi} holds, then
 $$\varrho \leq 
-\log \inf_{\lambda \geq 0}  \E[\exp(\lambda X)] .$$
\end{prop}
We note that the exponential moment assumption is equivalent to the existence of a proper rate function $\Lambda$ for the lower deviations of $S$. 
The lower bound (i) is proved in Section \ref{rho0}.
The second bound is the rate function of the random walk $S$ evaluated in $0$, say $\Lambda(0)$. Indeed, recalling that $\E[X]>0$, the supremum in the Legendre transform can be taken over $\R^+$ instead of $\R$. Exctracting $-1$ yields the upper bound above. It
 is proved in Section \ref{section6} and used for the proof of the next Corollary \ref{lf1}. It can be reached and  has a natural interpretation in terms of 
environmental stochasticity. Indeed, one way to keep the population bounded but alive  comes from a 'critical environment', which means $S_n \approx 0$.  Then $\E[Z_n \ \vert \ \mathcal{E}]=\exp(S_n)$ is neither small nor large and one can expect that the population is positive but bounded. 
The event $\{S_n \approx 0 \}$ is a large deviation event whose probability decreases exponentially with rate $\Lambda(0)$. This bound is thus directly explained by the \emph{environmental  stochasticity}.  \\

Now, we focus on the linear fractional case. We derive an explicit expression of $\varrho$ and describe the position of the most recent 
common ancestor of the population conditioned to be positive but small. We recall that a probability generating function of a random variable $R$ is linear fractional (LF) if there exist positive real numbers  $m$ and $b$ such that 
$$f(s)=1- \frac{1-s}{m^{-1}+  b m^{-2} (1-s)/2},$$
where $m=f'(1)$ and $b=f''(1)$. This family includes the probability generating function of geometric distributions, 
 with $b=2m^2$. More precisely, LF distributions are geometric laws with a second free parameter $b$  which allows to change the probability of 
the event $\{R=0\}$.

\begin{cor}\label{lf1}
We assume that $f$ is a.s. linear fractional,  $\mathbb{E}[|X|]<\infty$, $\mathbb{E}[X^2e^{-X}]<\infty$ and $\mathbb{P}(Z_1=0)>0$. \\
We assume also that either  $\P(X\geq 0)=1$ or Assumption \ref{axi} hold. Then
\begin{eqnarray}
\varrho&=& \left\{ \begin{array}{l@{\quad,\quad}l}
                         -\log\mathbb{E}\big[e^{-X}\big]     & \mbox{if } \ \mathbb{E}[Xe^{-X}]\geq 0\\
-\log \inf_{\lambda \geq 0} \E[\exp(\lambda X)]  \ \ (=\Lambda(0)) & \mbox{else}
                             \end{array} \right. .
\end{eqnarray}
\end{cor}
This result is also stated in the PhD of one of the authors and can  be found in \cite{boe1}.
On the level of large deviation ($\log$ scale), two regimes in the supercritical case are visible. \\
 If $\mathbb{E}[Xe^{-X}]< 0$, the event $\{ Z_n= k\}$ 
is a typical event in a suitable exceptional environment, say 'critical'. This rare event is then explained (only) by the environmental stochasticity.  \\
If $\mathbb{E}[Xe^{-X}]\geq 0$, we recover a term analogous to the Galton-Watson case, which is smaller than $\Lambda(0)$. 
The rare event is then  due to \textsl{demographical stochasticity}. \\
These two regimes seem to be analogous to the two regimes  in the subcritical case, which deal with the asymptotic behavior of $Z_n>0$, see e.g. \cite{Dek, GL, GKV03}.
Such regimes for supercritical branching processes have already  
been observed in  \cite{MH} in the continuous framework (which essentially represents linear fractional offspring-distributions).\\

Let us now focus on the  most recent common ancestor (MRCA) of the population conditionally on this rare event. More precisely, let $\mathcal{T}^n$ be
 the set of all ordered, rooted trees of height exactly $n$. We refer to  \cite{ne} for  classical definitions. 
 We say that an individual $i_{n_2}$ in generation $n_2>n_1$ stems from an individual $i_{n_1}$ 
iff there are individuals $i_{n_2-1},\ldots, i_{n_1+1}$ such that $i_{n_2}$ is a child of $i_{n_2-1}$, $i_{n_2-1}$ is a child of $i_{n_2-2}$, $\ldots$ and $i_{n_1+1}$ is a child of $i_{n_1}$.
Let $T_n\in\mathcal{T}_n$ be the random branching tree, generated by the process $(Z_k)_{0\leq k\leq n}$ conditioned on  $Z_n>0$ and denote by $MRCA_n$ the most recent common ancestor of the population at time $n$. More precisely, we consider the events 
$$A_k^n:=\big\{\text{ all individuals in generation $n$  stem from one individual in generation $n-k$}\}$$
and define the age of the MRCA in generation $n$ as the number of generations one has to go back in the past until all individuals in generation $n$ have a single common ancestor :
\begin{align*} 
MRCA_n:=\min\{k=1,2\ldots, n \ | A_k^n \ \  \text{holds} \ & \} .
\end{align*}

The case $\mathbb{P}(Z_1=0)=0$ is trivial :   extinction is  not possible, so  
 $\varrho=-\log \mathbb{P}_1(Z_1=1)$ and $\mathbb{P}_2(MRCA_n=n|Z_n=2)=1$. 
It is excluded in the statement below. Moreover, for the sake of simplicity, we restrict the study of the MRCA  to starting from one individual and conditioning on $Z_n=2$. 

\begin{cor}\label{corMRCA} We make the same assumptions as in the previous corollary. 

(i) If $\mathbb{E}[Xe^{-X}]<0$, then
\[\liminf_{n\rightarrow\infty} \mathbb{P}_1(MRCA_n=n| Z_n=2)> 0 \qquad ; \qquad \liminf_{n\rightarrow\infty} \mathbb{P}_1(MRCA_n=1| Z_n=2)>0. \]

(ii) If $\mathbb{E}[Xe^{-X}]=0$, then for every sequence $(x_n)_{n\in \N}$ such that $x_n \in [1,n]$ for every $n$, we have
\[\lim_{n\rightarrow\infty} \tfrac{1}{n} \log \mathbb{P}_1(MRCA_n=x_n |Z_n=2)=0. \]
Moreover, under the additional assumption  $\mathbb{E}\big[f''(1)\big]<\infty$,  there exist two positive finite constants $c,C$ such that for every $n\in \N$,
\[c\leq n\ \mathbb{P}_1(MRCA_{n}=n|Z_{n}=2) \leq C \]
and if $\mathbb{E}\big[f''(1)/(1-f(0))^2\big]<\infty$, then for every $\delta\in (0,1)$, there exist  two positive finite constants $c',C'$ such that for every $n\in \N$,
\[c'\leq n^{3/2}\mathbb{P}_1(MRCA_{n}=\lceil\delta n\rceil|Z_{n}=2) \leq C'. \]

(iii) If $\mathbb{E}[Xe^{-X}]>0$ and $\mathbb{E}[e^{(-1-s)X}]<\infty$ for some $s>0$, then for every $\delta\in(0,1]$,
\[\limsup_{n\rightarrow\infty}\tfrac 1n \log\mathbb{P}_1(MRCA_n>\delta n|Z_n=2)<0. \]
\end{cor}

Thus three regimes appear for  the most recent common ancestor of the population. \\
If $\mathbb{E}\big[Xe^{-X}\big]>0$, which we 
call the 'strongly' supercritical case, the MRCA is at the end (close to the actual time). The probability that the MRCA is far away from the final generations decreases 
exponentially. 
Such a result is classical for branching processes which do not explode, such as subcritical Galton-Watson processes conditioned on survival. 
It corresponds to a spine decomposition of the population whose subtrees become extinct \cite{LPP}. Conditionally  on $\{Z_n=2\}$, $S$ is 
still a random walk with positive drift and will be typically large. Thus 
the conditioned process is typically small throughout all generations (as in the Galton-Watson case) as growing and then becoming small again within the favorable environment 
 has a very small probability. Consequently, the MRCA will be close to generation $n$.\\
But in the 'weakly' supercritical case 
($\mathbf{E}\big[Xe^{-X}\big]<0$), conditionally  on $\{Z_n=2\}$, the MRCA is either at the beginning (close to the root of the tree) or at the end (close to generation $n$). Such a  situation is much less usual. 
It has already been observed in  \cite{reduced} for the subcritical reduced tree of linear fractional 
BPRE conditioned to survive. Here, as indicated in the proof of Proposition \ref{prop2} (ii),  the random walk $S$ conditioned on $\{ Z_n= 2 \}$ typically looks like an excursion. It means that $S$ is conditioned on the event $\{\min\{S_0,\ldots S_n\}\geq 0, S_n\leq c\}$. In such an environment, subtrees that are either born at the beginning or at 
the end may survive until the end. All subtrees being born at some generation $\lfloor \delta n\rfloor$, $\delta\in(0,1)$ experience an unfavorable environment 
and become extinct. 
This can be seen as follows. During an excursion from 0 to $n$, typically $S_{\lfloor \delta n\rfloor}\gg 0$ and thus $e^{S_n-S_{\lfloor \delta n\rfloor}}\ll0$ and the corresponding subtree will become extinct.\\
 Finally, in the intermediate case ($\mathbb{E}\big[Xe^{-X}\big]=0$), 
the MRCA is close to the end, but the probability that the MRCA is far away from the end only decreases polynomially.
The intermediately supercritical regime is in-between the two regimes described above and conditioned on $\{Z_n=2\}$, the typical environment will neither be an excursion nor a random walk with positive 
drift.\\
One can expect several more detailed results describing the three regimes, which are beyond the scope of this paper.

\section{The Geiger construction for a branching process in varying environment (BPVE)}\label{section4}

In this section, we work in a quenched environment, which means that we fix the environment $e:=(q_1,q_2,\ldots)$. We  consider  a 
branching process  in varying environment $e$
and denote by $\mathsf{P}(\cdot)$ (resp. $\mathsf{E}$) the associated probability (resp. expectation), i.e.
$$\mathsf{P}( Z_1 = k_1, \cdots ,Z_n=k_n )=\mathbb{P}(Z_1 = k_1, \cdots ,Z_n=k_n|\mathcal E =e).$$
Thus $(f_1,f_2,\ldots)$ is now fixed and the probability generating function of $Z$ is given by 
$$\mathsf{E}\big[s^{Z_{n}} \ \vert  \  Z_0=k \big]=f_{0,n}(s)^{k} \qquad (0\leq s\leq 1).$$

We use a construction of $Z$  conditioned on survival, 
which is due  to \cite{geiger99}[Proposition 2.1] and extends the spine  construction of Galton-Watson processes \cite{LPP}.
In each generation, the individuals 
 are labeled by the integers $i=1,2,\cdots$ in a breadth-first manner ('from the left to the right'). 
We follow then the  \lq ancestral line' of the leftmost individual having a descendant in generation $n$. This line is denoted by $\mathbb L$. It means that  in generation $k$, the descendance of the  individual labeled  $\mathbb{L}_k$ survives until time $n$, whereas all the individuals whose label is less than $\mathbb{L}_k$ 
become extinct before time $n$.
The Geiger construction ensures that to the left of $\mathbb{L}$, independent subtrees conditioned on extinction in generation $n$ are growing.
To the right of $\mathbb{L}$, independent (unconditioned) trees are evolving.  Moreover the joint distribution of $\mathbb{L}_k$ and the number of offsprings in 
generation $k$ is known (see e.g. \cite{abkv11} for details, where $L:=\mathbb{L}-1$) and for every $k\geq 1$, $z\geq 1$ and $1\le1 l\leq z$,
\be
\label{eqcoupl}
 \mathsf{P}(Z_k=z, \mathbb{L}_k=l|Z_{k-1}=1, Z_n>0) = q_k(z) \frac{\mathsf P(Z_n >0 \ | Z_k=1) \mathsf{P}(Z_n=0 | Z_k=1 )^{l-1}}{\mathsf{P}(Z_n >0 \ | Z_{k-1}=1)}.
\ee
\begin{figure}[here]
\begin{center}
\includegraphics[angle=0,width=0.9\textwidth]{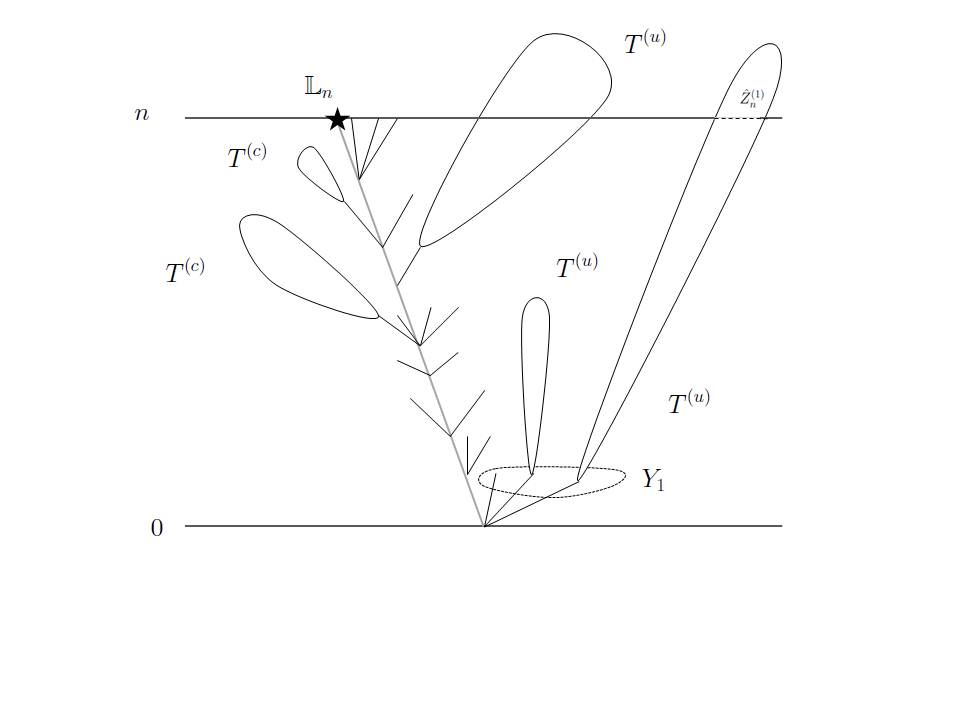}
\end{center}
\caption{\label{geiger1}  Geiger construction with $T^{(c)}$ trees conditioned on extinction and $T^{(u)}$ unconditioned trees.} 
\end{figure}

Let us give more details of this construction. We assume that the process starts with $Z_0=z$ and denote $\mathsf{P}_z(\cdot):=\mathsf{P}(\cdot|Z_0=z)$. 
We define for $0 \leq k <n$,
$$\mathsf{p}_{k,n}:= \mathsf{P}(Z_n>0 \ \vert \ Z_k=1)=1-f_{k,n}(0), \qquad \mathsf{p}_{n,n}:=1.$$
We can specify  the distribution of the number $Y_k$ of unconditioned trees founded by the ancestral line in generation $k$, at the right of $\mathbb{L}_k$.
In generation $0$, for $0\leq i\leq z-1$, 
\begin{align}
 \mathsf{P}_z(Y_0=i |Z_n>0)&:= \mathsf{P}_{z} (\mathbb{L}_0 =z-i \ \vert \ Z_n >0) =   \frac{\mathsf P(Z_n >0 \ | Z_0=1) \mathsf{P}(Z_n=0 | Z_0=1 )^{z-i-1}}{\mathsf{P}(Z_n >0 \ | Z_{0}=z)}\nonumber\\
&= \frac{1-f_{0,n}(0)}{\mathsf{P}_z(Z_n>0)} f_{0,n}(0)^{z-i-1}.\label{subtreegen1}
\end{align}
More generally, for all $1\leq k\leq n$ and $i\geq 0$, (\ref{eqcoupl}) yields
 \begin{align}
 \mathsf{P}(Y_k=i |Z_n>0)&:=
\mathsf{P}(Z_k-\mathbb{L}_k=i|Z_n>0, Z_{k-1}=1)\nonumber \\
&= \sum_{j=i+1}^\infty \mathsf{P}(Z_k=j, \mathbb{L}_k=j-i|Z_n>0, Z_{k-1}=1) \nonumber \\
&= \frac{\mathsf{p}_{k,n} }{\mathsf{p}_{k-1,n}}\sum_{j=i+1}^\infty q_k(j) f_{k,n}(0)^{j-i-1}. \label{conddis2}
\end{align}
Finally, we note  that $f_{n,n}(0)=0$, thus for  $k=n$, we have $\mathsf{P}(Y_n=i |Z_n>0)=\tfrac{q_n(i+1)}{p_{n-1,n}}$.\medskip\\
Here, we do not require the full 
description of the conditioned tree since we are only interested in the population alive  at time $n$. Thus we do not have to consider the trees conditioned on extinction,
 which grow to the left of $\mathbb{L}$. We can  construct the population alive in generation $n$  using the i.i.d random variables  $\hat{Y}_0,\hat{Y}_1,\hat{Y}_2, \ldots, \hat{Y}_n$ whose distribution is specified by (\ref{subtreegen1}) and (\ref{conddis2}) :
$$\P( \hat{ Y}_k =i)= \mathsf{P}(Y_k=i |Z_n>0)$$
Let 
$(\hat{Z}_j^{(k)})_{j\geq 0}$ be  independent branching processes in varying environment which are distributed as $Z$ for $j >k$ and satisfy
$$ \hat{Z}_j^{(k)} := 0 \t{ for}\quad j<k, \qquad  \hat{Z}_k^{(k)}:= \hat{Y}_k.$$ 
More precisely, for all $0\leq k \leq n$ and $z_0, \cdots ,z_n \geq 0$, 
\Bea
&&\mathsf{P}(\hat{Z}_0^{(k)}=0, \cdots, \hat{Z}_{k-1}^{(k)}=0, \hat{Z}_k^{(k)}=z_k, \hat{Z}_{k+1}^{(k)}=z_{k+1} , \cdots,  \hat{Z}_{n}^{(k)}=z_{n}) \\
&& \qquad \qquad \qquad \qquad =\mathsf{P}(\hat{Y}_k=z_k) \mathsf{P}(Z_{k+1}=z_{k+1}, \cdots,  Z_n=z_n \ \vert \ Z_k=z_k).
\Eea
The  sizes of the independent subtrees  generated  by the ancestral line in generation $k$, 
which may survive until generation $n$, are given by $(\hat{Z}_j^{(k)})_{0\leq j\leq n}$, $0\leq k\leq n-1$. In particular,
\begin{align}
 \mathcal{L}(Z_n|Z_n>0) = \mathcal{L}(\hat{Z}_n^{(0)} + \ldots+\hat{Z}_n^{(n-1)}+\hat{Y}_n+1). \label{geigerdecomp}
\end{align}
\begin{lemma}\label{le_subtree1}
 The probability that all  subtrees  emerging before generation $n$ become extinct before generation $n$ is given for $z\geq 1$ by 
\begin{align}
 \mathsf{P}_z(\hat{Z}_n^{(0)} + \ldots+\hat{Z}_n^{(n-1)}=0)&= \prod_{k=0}^{n-1}\mathsf{P}_z(\hat{Z}_n^{(k)}=0) =\frac{\mathsf{p}_{n-1,n}}{\mathsf{p}_{-1,n}} \prod_{k=0}^{n-1}  f'_k(f_{k,n}(0)), \nonumber 
\end{align}
where we use the following convenient notation $f_0(s):=s^z$, $p_{-1,n}:=\mathsf{P}_z(Z_n>0)$.
\end{lemma}
\begin{proof}
First, we compute  the probability that 
the subtree generated by the ancestral line in generation $k$ does not survive until generation $n$, i.e. $\{\hat{Z}_n^{(k)}=0\}$. By (\ref{conddis2}), for $k\geq 1$,
\bea
 \mathsf{P}(\hat{Z}_n^{(k)}=0)&=& \sum_{i=0}^\infty  \mathsf{P}_z(\hat{Y}_k=i ) \mathsf{P}(Z_n=0|Z_{k}=i)\nonumber\\
&=& \frac{\mathsf{p}_{k,n}}{\mathsf{p}_{k-1,n}} \sum_{i=0}^\infty \sum_{j=i+1}^\infty q_k(j)f_{k,n}(0)^{j-i-1}\cdot f_{k,n}(0)^{i}\nonumber\\
&=&\frac{\mathsf{p}_{k,n}}{\mathsf{p}_{k-1,n}} \sum_{i=0}^\infty \sum_{j=i+1}^\infty q_k(j)f_{k,n}(0)^{j-1}\nonumber \\
&=& \frac{\mathsf{p}_{k,n}}{ \mathsf{p}_{k-1,n}}\sum_{j=1}^\infty j  q_k(j) f_{k,n}(0)^{j-1}\nonumber \\
&=& \frac{\mathsf{p}_{k,n}}{ \mathsf{p}_{k-1,n}} f'_k(f_{k,n}(0)). \label{equachap} 
\eea
Similarly,  we get from (\ref{subtreegen1}) that
\bea 
 \mathsf{P}_z(\hat{Z}_n^{(0)}=0)&=& \sum_{i=0}^{z-1}  \mathsf{P}_z(Y_0=i|Z_n>0 ) \mathsf{P}(Z_n=0|Z_{0}=i)\nonumber\\
&=& \sum_{i=0}^{z-1}  \frac{1-f_{0,n}(0)}{\mathsf{P}_z(Z_n>0)} f_{0,n}(0)^{z-i-1} f_{0,n}(0)^i \nonumber\\
&=&  \frac{\mathsf{p}_{0,n}}{ \mathsf{p}_{-1,n}} zf_{0,n}(0)^{z-1}=\frac{\mathsf{p}_{0,n}}{ \mathsf{p}_{-1,n}} f'_0(f_{0,n}(0)) \nonumber
\eea
with the convention $f_0(s)=s^z$. Adding that the subtrees given by $(\hat{Z}_j^{(k)})_{j\geq 0}$ are independent, a telescope argument yields the claim. 
\end{proof}
For the next lemma, we introduce  the last generation before $n$ when the environment allows extinction :
$$\kappa_n:=\sup\{1\leq k\leq n \ : \  q_k(0)>0\}, \qquad (\sup \varnothing =0).$$
Note that $\kappa_n$ only depends on the environment up to generation $n$.
\begin{lemma}\label{le_subtree2}
Let $z_0$ be the smallest element in $\mathcal{I}$. Then, 
\begin{align}
 \mathsf{P}_{z_0}(Z_n=z_0 \ \vert \ Z_n>0)= \frac{q_{\kappa_n}(z_0)}{\mathsf{p}_{\kappa_n-1,\kappa_n}} \times \prod_{k=0}^{\kappa_n-1}\frac{\mathsf{p}_{k,\kappa_n}}{\mathsf{p}_{k-1,\kappa_n}}  f'_k(f_{k,\kappa_n}(0))  \times \prod_{j=\kappa_n+1}^{n} q_j(1)^{z_0}, \nonumber
\end{align}
where we recall  the following convenient notation $f_0(s)=s^z$, $p_{-1,n}=\mathsf{P}_z(Z_n>0)$.
\end{lemma}
\begin{proof}
By definition of $\mathcal{I}$ and $z_0$, $q(0)>0$ implies 
$q(k)=0$ for every  $1\leq k< z_0$. 
We first deal with the case $\kappa_n>0$. Then,
$$q_{\kappa_n}(0)>0, \quad q_{\kappa_n}(k)=0 \ \t{if} \ 1\leq k < z_0 \ ;  \qquad q_{\kappa_n+1}(0)=\cdots =q_{n}(0)=0.$$
In particular the number of individuals in generation $\kappa_n$ is at least $z_0$ times the number of individuals in generation $\kappa_{n}-1$  
who leave at least one  offspring in generation $\kappa_n$. Moreover, as extinction is not possible after generation $\kappa_n$, it holds that $Z_{\kappa_n}\leq Z_{\kappa_n+1}\leq \cdots\leq Z_{n}$. \\
Let us consider the event  $Z_n=z_0>0$. Then $Z_{\kappa_n-1}>0$ and  $Z_{\kappa_n} \geq z_0$. So  $Z_{\kappa_n}=Z_{\kappa_n+1}=\cdots=Z_{n}=z_0$ and 
only a single  individual in generation $\kappa_n-1$ leaves  one offspring (or more) in generation $\kappa_n$. This individual lives on the ancestral line. Thus all the subtrees to the right of the ancestral line which are born before generation $\kappa_n$ have become 
extinct before 
generation $\kappa_n$, i.e. $\hat{Z}_{\kappa_n}^{(0)}= \ldots=\hat{Z}_{\kappa_n}^{(\kappa_n-1)}=0$. In generation $\kappa_n-1$, the individual on the
ancestral line has $z_0$ offsprings and $\hat{Y}_{\kappa_n}=z_0-1$. After generation $\kappa_n$, all the individuals must leave exactly one offspring  to  keep the population constant until generation $n$, since 
$q_{\kappa_n+1}(0)=\cdots =q_{n}(0)=0$. This probability is then given by $q_j(1)^{z_0}$ in generation $j>\kappa_n$. 
Moreover, $(\ref{conddis2})$ simplifies to $\mathsf{P}(\hat{Y}_{\kappa_n}=z_0-1) = q_{\kappa_n}(z_0)/\mathsf{p}_{\kappa_n-1,\kappa_n}$.
\begin{figure}[here]
\setlength{\unitlength}{1cm}
\begin{center}
\includegraphics[angle=0,width=0.5\textwidth]{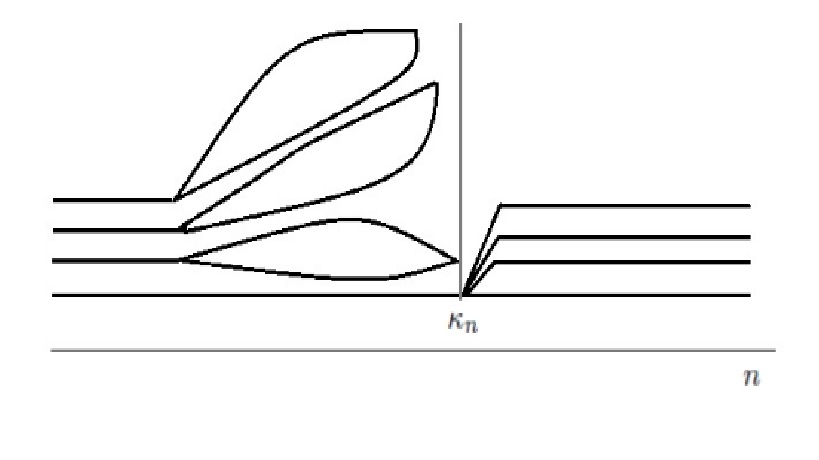}
\end{center}
\caption{\label{geiger2} Illustration of the proof of Lemma \ref{le_subtree2}.} 
\end{figure}
Using the previous lemma, it can be written as follows:
\begin{align*}
& \mathsf{P}_{z_0}(Z_n=z_0) \\
& \qquad  = \mathsf{P}_{z_0}(\hat{Z}_{\kappa_n}^{(0)}= \ldots=\hat{Z}_{\kappa_n}^{(\kappa_n-1)}=0) \mathsf{P}(\hat{Y}_{\kappa_n}=z_0-1)\mathsf{P}_{z_0}(\hat{Z}_{n}^{(\kappa_n)} + \ldots+\hat{Z}_{n}^{(n-1)}+Y_n+1=z_0) \\
& \qquad =   \left[\prod_{k=0}^{\kappa_n-1}\frac{\mathsf{p}_{k,n}}{\mathsf{p}_{k-1,n}}  f'_k(f_{k,\kappa_n}(0)) \right]\frac{q_{\kappa_n}(z_0)}{\mathsf{p}_{\kappa_n-1,\kappa_n}}\mathsf{P}_{z_0}(\hat{Z}_{n}^{(\kappa_n)} + \ldots+\hat{Z}_{n}^{(n-1)}+Y_n+1=z_0) \\
& \qquad =  \frac{q_{\kappa_n}(z_0)}{\mathsf{p}_{\kappa_n-1,\kappa_n}}\left[\prod_{k=0}^{\kappa_n-1}\frac{\mathsf{p}_{k,n}}{\mathsf{p}_{k-1,n}}  f'_k(f_{k,\kappa_n}(0)) \right]\prod_{j=\kappa_n+1}^{n} q_j(1)^{z_0}. 
\end{align*}
Recall that after generation $\kappa_n$, each individual has at least one offspring and thus $\mathsf{p}_{j,n}=\mathsf{p}_{j,\kappa_n}$ for any $j<\kappa_n$. This  ends up the proof in the case $\kappa_n>0$. The case when 
 $\kappa_n=0$ is easier. Indeed,
$$ \mathsf{P}_{z_0}(Z_n=z_0)= \mathsf{P}_{z_0}(Z_1=\cdots=Z_n=z_0)= \prod_{j=1}^{n} q_j(1)^{z_0}$$ since $q_{\kappa_n+1}(0)=\cdots =q_{n}(0)=0$ and  $Z$ is nondecreasing until generation $n$.  
\end{proof}

\section{Proof of Theorem \ref{theo_rho1}  : the probability of staying positive but bounded}\label{proofs}
In this section, we  prove Theorem \ref{theo_rho1} with the help of two  lemmas. The first lemma establishes the existence of a proper 'common' limit.
\begin{lemma}\label{rho_1}
 Assume that $z\geq 1$ satisfies $\mathbb{P}(Q(0)>0, Q(z)>0)>0$. \\
Then for all $k,j\in Cl(\{z\})$,
 the following limits exist in $[0,\infty)$ 
and coincide  
\begin{align} 
\lim_{n\rightarrow\infty} \tfrac{1}{n} \log \mathbb{P}_k(Z_n=j)= \lim_{n\rightarrow\infty} \tfrac{1}{n}\log\mathbb{P}_{z}(Z_n=z). \nonumber
\end{align}
Moreover, for every sequence $k_n$ such that $k_n \geq z$ for $n$ large enough and $k_n/n\rightarrow 0$ as $n\rightarrow \infty$,
\begin{align*}
 \lim_{n\rightarrow\infty} \tfrac{1}{n} \log \mathbb{P}_{z}(Z_n=z)= \lim_{n\rightarrow\infty} \tfrac{1}{n}\log\mathbb{P}_{z}(1\leq Z_n\leq k_n).
\end{align*}
\end{lemma}
\textsl{Proof.}
Note that for every $k\geq 1$, $\mathbb{P}_{k}(Z_1=z)>0$ since 
$$\mathbb{P}_{k}\big(Z_1=z   \ \vert \ Q_1\big)\geq Q_1(0)^{k-1} Q_1(z), \qquad
\mathbb{P}(Q(0)>0, Q(z)>0)>0.$$
 We know that by Markov property, for all $m,n\geq 1$,
\begin{align}
 \mathbb{P}_{z}(Z_{n+m}=z)& 
\geq \mathbb{P}_{z}(Z_{n}=z) \mathbb{P}_{z}(Z_m=z). \label{help01}
\end{align}
Adding that $\mathbb{P}_{z}(Z_{1}=z)>0$, we obtain that  the sequence $(a_n)_{n\in\mathbb{N}}$ defined by $a_n:=-\log \mathbb{P}_{z}(Z_n=z)$ is finite and subadditive. Then Fekete's lemma ensures that
$\lim_{n\rightarrow\infty}  a_n/n$ exists and belongs to $[0,\infty)$. 
 Next, if $j,k \in Cl(\{z\})$, there exist $l,m \geq 0$ such that
 $\mathbb{P}_{z}(Z_{l}=j)>0$ and $\mathbb{P}_{z}(Z_{m}=k)>0.$  
We get
$$ \mathbb{P}_k(Z_{n+l+1}=j)\geq \P_k(Z_1=z)\P_{z}(Z_{n}=z)\P_{z}(Z_l=j)$$
and
$$\P_{z}(Z_{m+n+1}=z) \geq\P_{z}(Z_m=k)\mathbb{P}_k(Z_{n}=j)\P_j(Z_1=z).$$
Adding that  $\P_j(Z_1=z)>0$, we obtain 
$$\liminf_{n\rightarrow\infty} \tfrac{1}{n} \log \mathbb{P}_k(Z_n=j)\geq  \lim_{n\rightarrow\infty} \tfrac{1}{n}\log\mathbb{P}_{z}(Z_n=z)\geq \limsup_{n\rightarrow\infty} \tfrac{1}{n} \log \mathbb{P}_k(Z_n=j),$$
which yields the first result.\medskip\\
For the second part of the lemma, we first  observe that  $\P_{z}(Z_n=z) \leq \P_{z}(1\leq Z_n\leq k_n)$ for $n$ large enough. To prove the converse inequality, we
define for $\varepsilon>0$ the set
\begin{align*}
  \mathcal{A}_{\varepsilon}:= \{ q\in\Delta | q(0)>\varepsilon,\, q(z)>\varepsilon\} .
\end{align*}
According to the definition of $\mathcal{I}$ and the assumption, $\mathbb{P}(Q\in\mathcal{A}_{\varepsilon})>0$ if $\varepsilon$ is chosen small enough. Thus, we get  
\begin{align*}
 \P_{z}(Z_n=z)&\geq \P_{z}(1\leq Z_{n-1} \leq k_n) \min_{1\leq j\leq k_n} \P_{j}(Z_1=z) \\
&\geq \P_{z}(1\leq Z_{n-1} \leq k_n) \mathbb{P}(Q\in\mathcal{A}_{\varepsilon})  \min_{1\leq j\leq k_n} \E[\P_1(Z_1=z)\P_1(Z_1=0|Q)^{j-1} |Q\in\mathcal{A}_{\varepsilon}] \\
&\geq \P_{z}(1\leq Z_{n-1} \leq k_n) \mathbb{P}(Q\in\mathcal{A}_{\varepsilon})\  \varepsilon^{k_n}.
\end{align*}
Using the logarithm of this expression and letting $n\rightarrow \infty$ yields 
\begin{align*}
 \lim_{n\rightarrow\infty} \tfrac 1n \log \P_{z}(Z_n=z) &\geq \limsup_{n\rightarrow\infty} \left( \tfrac 1n \log \P_{z}(1\leq Z_{n-1} \leq k_n) + \log(\varepsilon) \tfrac{k_n}{n} \right).
\end{align*}
Adding that $k_n=o(n)$ by assumption gives the claim.\qed\medskip\\

Now, we prove a representation of the limit $\rho$ in terms of generating functions. It will be useful in the rest of the paper.  
\begin{lemma}\label{lemchar}
 Assume that $\mathbb{P}_1(Z_1=0)>0$. Then for all $i,j \in Cl(\mathcal{I})$,
$$
 \lim_{n\rightarrow\infty} \tfrac{1}{n} \log \mathbb{P}_{i}(Z_n=j)
 =\lim_{n\rightarrow\infty} \tfrac{1}{n} \log \E\Big[Q_{n}(z_0)f_{0,n}(0)^{z_0-1} \prod_{i=1}^{n-1}f'_i\big(f_{i,n}(0)\big)\Big],$$
where $z_0$ is the smallest element in $\mathcal{I}$.
\end{lemma}
We note that $\mathbb{P}_1(Z_1=0)>0$ is equivalent to  $\mathbb{P}(Q(0)>0)>0$ and in view of Lemma \ref{rho_1}, we only have to prove the result for $k=j=z_0$, where $z_0$ is the smallest element in $\mathcal{I}$. Differentiation of the probability generating function of $Z_n$ yields 
the result for $z_0=1$. The generalization of the result for $z_0 \ne 1$ via higher order derivatives of generating functions appears to be complicated. Instead,
we use probabilistic arguments, involving the Geiger construction of the previous section.
\begin{proof}
First, the result is obvious  when $z_0=1 \in \mathcal{I}$ since
\begin{align}
 \mathbb{P}(Z_n=1|\mathcal{E}) = \frac{d}{ds} f_{0,n}(s) \big|_{s=0} = f'_n(0) \cdot\prod_{i=1}^{n-1} f'_{i}(f_{i,n}(0)). \nonumber
\end{align}
For the case $z_0>1$, we start by proving the lower bound.
Using (\ref{geigerdecomp}), Lemma \ref{le_subtree1},  (\ref{subtreegen1}) and  recalling that $\mathbb{P}_{z_0}(Z_n>0|\mathcal{E})=p_{-1,n}$,  we have 
\begin{align}
 \mathbb{P}_{z_0}(Z_n=z_0) &= \mathbb{E}\Big[\mathbb{P}_{z_0}(Z_n=z_0|Z_n>0,\mathcal{E})\mathbb{P}_{z_0}(Z_n>0|\mathcal{E}) \Big]\nonumber \\
&=\mathbb{E}\Big[\mathbb{P}_{z_0}(\hat{Z}_n^{(0)} + \cdots + \hat{Z}_{n}^{(n-1)}+\hat Y_n+1=z_0|\mathcal{E})\mathbb{P}_{z_0}(Z_n>0|\mathcal{E}) \Big]\nonumber \\
&\geq \mathbb{E}\Big[\mathbb{P}_{z_0}(\hat{Z}_n^{(0)} + \cdots + \hat{Z}_{n}^{(n-1)}=0, \hat Y_n=z_0-1|\mathcal{E})\mathbb{P}_{z_0}(Z_n>0|\mathcal{E})\Big] \nonumber \\
&= \mathbb{E}\Big[\mathbb{P}_{z_0}(Z_n>0|\mathcal{E})\  \frac{Q_{n}(z_0)}{\mathsf{p}_{n-1,n}} \ \frac{\mathsf{p}_{n-1,n}}{\mathsf{p}_{-1,n}}  \prod_{i=0}^{n-1} f'_i(f_{i,n}(0)) \Big] \nonumber \\
&= \mathbb{E}\Big[Q_{n}(z_0) \prod_{i=0}^{n-1} f'_i\big(f_{i,n}(0)\big) \Big] .\label{interpretation}
\end{align}
Recalling also that  $f'_0(s)=z_0 s^{z_0-1}$, we get
\begin{align*}
 \lim_{n\rightarrow\infty} \tfrac 1n \log \mathbb{P}_{z_0}(Z_n=z_0) &\geq  \limsup_{n\rightarrow\infty} \tfrac 1n \log \mathbb{E}\Big[Q_{n}(z_0)f_{0,n}^{z_0-1}(0) \prod_{i=1}^{n} f'_i\big(f_{i,n}(0)\big) \Big].
\end{align*}
Let us now prove the converse inequality. Following the previous section, $z_0$ is the smallest element in $\mathcal{I}$
and  $\kappa_n$ is the (now random) last moment when an environment satifies  $Q(0)>0$.
We  decompose the event $\{Z_n=z_0\}$ according to $\kappa_n$ : 
$$  \mathbb{P}_{z_0}(Z_n=z_0)
=\sum_{k=0}^{n} \mathbb{E}\Big[\mathbb{P}_{z_0}(Z_n=z_0|\mathcal{E}, Z_n>0)\mathbb{P}_{z_0}(Z_n>0|\mathcal{E}); \kappa_n=k \Big]
$$
Using  that conditionally on $\kappa_n=k$, $\mathbb{P}_{z_0}(Z_n>0|\mathcal{E})=\mathbb{P}_{z_0}(Z_k>0|\mathcal{E})$ and Lemma \ref{le_subtree2}, we get by independence
\begin{align}
 \mathbb{P}_{z_0}(Z_n=z_0)&=\sum_{k=0}^{n} \mathbb{E}\Big[ \mathbb{P}_{z_0}(Z_k>0|\mathcal{E}) \frac{Q_k(z_0)}{\mathsf{p}_{k-1,k}}\frac{\mathsf{p}_{k-1,k}}{\mathsf{p}_{-1,k}}\prod_{i=0}^{k-1}  f'_i(f_{i,k}(0))  \prod_{j=k+1}^{n} Q_j(1)^{z_0}\ ;\ \kappa_n=k \Big]\nonumber\\
&\leq \sum_{k=0}^{n} \mathbb{E}\Big[ Q_k(z_0)\prod_{i=0}^{k-1}  f'_i(f_{i,k}(0))\Big] \prod_{j=k+1}^{n} \mathbb{E}\big[Q_j(1)^{z_0}\big]\ \nonumber \\
&= \sum_{k=1}^{n} \mathbb{E}\Big[ Q_k(z_0)\prod_{i=0}^{k-1} f'_i(f_{i,k}(0))\Big] \mathbb{E}\big[Q(1)^{z_0}\big]^{n-k-1}+ \mathbb{E}\big[Q(1)^{z_0}\big]^n.\nonumber 
\end{align}
Let $(a_n)_{n\in\N}$ be a sequence in $\R^+$ and $b>0$. Then, by standard results on the exponential rate of sums, we have  
\begin{align*}
\limsup_{n\rightarrow\infty}& \tfrac 1n \log \sum_{k=1}^n a_k b^{n-k}= \max\Big\{\limsup_{n\rightarrow\infty} \tfrac 1n \log a_n,\log b \Big\}.
\end{align*}
Thus
\Bea
&&  \lim_{n\rightarrow\infty} \tfrac 1n \log \mathbb{P}_{z_0}(Z_n=z_0)  \\
&& \qquad  \leq \max\Big\{\limsup_{n\rightarrow\infty} \tfrac{1}{n} \log \E\big[Q_{n}(z_0) \prod_{i=0}^{n}f'_i(f_{i,n}(0)\big)\big]
; \log\mathbb{E}\big[Q(1)^{z_0}\big]\Big\}.
\Eea
We now prove  that the first term always realizes the maximum. Using that $f'_0(f_{0,n}(0))=z_0f_{0,n}^{z_0-1}(0)=z_0 \mathbb{P}_1(Z_n=0|\mathcal{E})^{z_0-1}$ and
$f'_i\big(f_{i,n}(0)\big)\geq f'(0)$, we have
\begin{align*}
 \mathbb{E}\Big[Q_{n}(z_0) \prod_{i=0}^{n-1} f'_i\big(0\big) \Big] 
&\geq z_0 \mathbb{E}\Big[Q_{n}(z_0) \mathbb{P}_1(Z_n=0|\mathcal{E})^{z_0-1}\prod_{i=1}^{n-1} Q_i(1) \Big]\nonumber \\
&\geq z_0 \mathbb{E}\Big[Q_{n}(z_0) \Big(Q_n(0)\prod_{i=1}^{n-1} Q_i(1) \Big)^{z_0-1}\prod_{i=1}^{n-1} Q_i(1) \Big]\nonumber \\
&\geq z_0 \mathbb{E}\Big[Q_{n}(z_0) Q_n(0)^{z_0-1}\Big] \mathbb{E}\big[Q(1)^{z_0}\big]^{ n},\nonumber 
\end{align*}
By definition of $z_0$, $\mathbb{E}\big[Q_{n}(z_0) Q_n(0)^{z_0-1}]>0$, and we can conclude 
$$
 \lim_{n\rightarrow\infty} \tfrac 1n \log \mathbb{P}_{z_0}(Z_n=z_0) \leq \limsup_{n\rightarrow\infty} \tfrac{1}{n} \log \E\big[Q_{n}(z_0)f_{0,n}(0)^{z_0-1} \prod_{i=1}^{n}f'_i(f_{i,n}(0)\big)\big]\nonumber
$$
We end up the proof by checking the convergence of the sequence on the right-hand side above.
We use again (\ref{geigerdecomp}) and  Lemma \ref{le_subtree1} to write
\begin{align}
\phi_n:&=\mathbb{P}_{z_0}(\hat Z^{(0)}_n+\ldots+\hat Z^{(n-1)}_n=0, Z_n=z_0) \nonumber \\
&=\E\Big[\mathbb{P}_{z_0}(\hat Z^{(0)}_n+\ldots+\hat Z^{(n-1)}_n=0 \ \vert \EE)\P(\hat{Y}_n=z_0-1) \Big]\nonumber \\
&= \E\Big[Q_{n}(z_0)\prod_{i=0}^{n}f'_i\big(f_{i,n}(0)\big)\Big]. \label{interpretation2}
\end{align}
It is  the probability of having $z_0$-many individuals in generation $n$, where all individuals in 
generation $n$ have a common ancestor in generation $n-1$. By Markov property, for $k=1,\ldots,n$
\begin{align*}&\mathbb{P}_{z_0}(\hat Z^{(0)}_n+\ldots+\hat Z^{(n-1)}_n=0, Z_n=z_0) \\
& \qquad \geq \mathbb{P}_{z_0}(\hat Z^{(0)}_k+\ldots+\hat Z^{(k-1)}_k=0, Z_k=z_0)\mathbb{P}_{z_0}(\hat Z^{(0)}_{n-k}+\ldots+\hat Z^{(n-k-1)}_{n-k}=0, Z_{n-k}=z_0). 
\end{align*}
The same subadditivity arguments as in the proof of Lemma \ref{rho_1} applied to $\phi_n$ yield the existence of the limit of 
$\tfrac 1n \log  \phi_n$ and
\be 
\lim_{n\rightarrow\infty} \tfrac 1n \log \mathbb{P}_{z_0}(Z_n=z_0)= \lim_{n\rightarrow \infty} \tfrac{1}{n} \log \phi_n =  \lim_{n\rightarrow \infty} \tfrac{1}{n} \log \E\Big[Q_{n}(z_0)f_{0,n}(0)^{z_0-1} \prod_{i=1}^{n}f'_i\big(f_{i,n}(0)\big)\Big] \label{equat}
\ee
This ends up the proof.

\end{proof}

\section{Proof of  Proposition \ref{prop2}}
\subsection{Preliminaries on random walks}
In this section, we will shortly present some standard results on random walks which we use. In all following results, we take $S_0=0$. 
We assume that there exists   $s>0$ such that  $\mathbb{E}[Xe^{-sX}]=0$. This suggests to change to a measure $\mathbf{P}$, 
defined by 
\[ \mathbf{P}(X\in dx):= \frac{e^{-s x} \mathbb{P}(X\in dx) }{\mu}, \]
where $\mu:=\mathbb{E}[e^{-sX}]$. We note that $\mathbf{E}[X]=\mu^{-1} \mathbb{E}[Xe^{-sX}]=0$, and that $S$ is a recurrent random walk  under $\mathbf{P}$. 
In the following proofs, we use the change of measure described here. We define
\[L_n:=\min\{S_1,\ldots, S_n\} \] 
and 
\[M_n:=\min\{S_1,\ldots, S_n\}. \] 

The following result is directly derived from  
\cite{abkv10}[Proposition 2.1]. For lattice random walks, the claims result e.g. from \cite{VV}[Theorem 6].
\begin{lemma}\label{le_as1}
Assume that $\mathbf{E}[X]=0$ and $\mathbf{Var}(X)<\infty$. Then for every $\theta>0$, there exists   $d=d(\theta)$ such that
\[ \mathbf{E}[e^{-\theta S_n};L_n\geq 0] \sim d \ n^{-\tfrac 32} \quad (n\rightarrow\infty) \]
and
\[ \mathbf{E}[e^{\theta S_n};M_n< 0] \sim d \ n^{-\tfrac 32} \quad (n\rightarrow\infty) . \]
\end{lemma}   
The following lemma results from \cite{abkv10}[Equation (2.5) therein].
\begin{lemma}\label{le_as}
Assume that $\mathbf{E}[X]=0$ and $\mathbf{Var}(X)<\infty$. Then for every $c>0$ large enough, there exists   $d=d(c)$ such that
\[ \mathbf{P}(L_n\geq 0, S_n\leq c) \sim  d \ n^{-\tfrac 32} \qquad (n\rightarrow\infty). \]
\end{lemma} 
\textbf{Remark:} In the non-lattice case, the result holds for every $c>0$. In the lattice case, $c$ must be chosen such that $\mathbf{P}(0\leq S_1\leq c)>0$. \medskip\\

From the previous results, it follows that
\begin{cor}\label{corexp}
Assume that $\mathbf{E}[X]=0$ and $\mathbf{Var}(X)<\infty$. Then for every $\theta>1$,
\[\mathbf{E}\big[e^{-S_n+\theta L_n}\big]=O(n^{-3/2}).\]
\end{cor}
\begin{proof} We use a decomposition according to the first minimum of the random walk, i.e. let
\[\tau_n:=\min\{k\in\{0,\ldots,n\}\ | \ S_k=L_n\wedge 0\}. \]
Decomposing at the first minimum and using duality yields
\begin{align*}
 \mathbf{E}\big[e^{-S_n+\theta L_n}\big]&= \sum_{k=0}^n \mathbf{E}\big[e^{-(S_n-\theta S_k)};\tau_n=k\big]\\
&=\sum_{k=0}^n\mathbf{E}\big[e^{-(S_n-S_k)} e^{(\theta-1)S_k} ;\tau_k=k, \min_{i=k,\ldots,n} \{S_n-S_i\}\geq 0\big]\\
&=\sum_{k=0}^n\mathbf{E}\big[ e^{(\theta-1)S_k} ;\tau_k=k\big]\mathbf{E}\big[e^{-S_{n-k}}; L_{n-k}\geq 0\big]\\
& =\sum_{k=0}^n\mathbf{E}\big[ e^{(\theta-1)S_k} ;M_k<0\big]\mathbf{E}\big[e^{-S_{n-k}}; L_{n-k}\geq 0\big] .
\end{align*}
The last step follows from duality (see e.g. \cite{abkv11}). 
Recall that by assumption, $\theta-1>0$. Applying Lemma \ref{le_as1} now yields that there is a $c<\infty$ such that for $n$ large enough
\begin{align*}
 \mathbf{E}\big[e^{-S_n+\theta L_n}\big]& \leq c\Big(\frac{1}{n^{3/2}}+  \sum_{k=1}^{n-1} \frac{1}{(n-k)^{3/2} k^{3/2}}\Big) \\
&\leq  c\Big(\frac{2d}{n^{3/2}}+ \frac{2d}{\lfloor n/2\rfloor^{3/2} \sum_{k=0}^{\lfloor n/2\rfloor} k^{3/2}}\Big) =O(n^{-3/2}),
\end{align*}
which is the claim of the corollary.
\end{proof}

In the next lemma, we will use probability measures $\mathbf{P}^+$ and $\mathbf{P}^-$ which e.g. have been introduced in \cite{abkv10}. Here, we recall the definition.
Define the renewal functions $u: \mathbb{R} \to \mathbb{R}$ and $v: \mathbb{R} \to \mathbb{R}$ by
\begin{align*} u(x)  \ &= \ 1 + \sum_{k=1}^\infty \mathbf{P} (-S_k\le x, \max\{S_1,\ldots,S_k\} < 0 ), \quad x \ge 0,\\
v(x) \ &:= \ 1 + \sum_{k=1}^\infty \mathbf{P} (-S_k > x, \min\{S_1,\ldots,S_k\} \geq 0), \quad x \leq 0 \  ,\\
v(0) \ &= \ u(0) \  = \ 1 ,
\end{align*}
and 0 elsewhere. Using the identities 
\begin{equation} \begin{array}{rl} \mathbf{E}  [u(x+X); X + x \ge 0] \ = \ u(x),  &x \ge 0,  \\   \mathbf{E} [ v(x+X);X+x<0] \ = \ v(x),   &x \le 0, \end{array}  \label{harm}
\end{equation}
which hold for every oscillating random walk, one can construct probability measures $\mathbf{P}^+$ and $\mathbf{P}^-$: 
Define the filtration $\mathcal F=(\mathcal F_n)_{n \ge 0}$, where $\mathcal F_n=\sigma(Q_1,\ldots,Q_n,Z_0,\ldots,Z_n)$ . Then $S$ is adapted to 
$\mathcal F$ and $X_{n+1}$ (as well as $Q_{n+1}$) is independent of $\mathcal F_n$ for all $n \ge 0$. Now, for every bounded, $\mathcal F_n$-measurable random variable $R_n$
we can define
\begin{align*} \mathbf E^+_x[ R_n ] \ &= \ \frac{1}{u(x)}\mathbf{E}_x[  R_n u(S_n); L_n \ge 0], \quad x \ge 0,\\ \mathbf E^-_x[ R_n ] \ &= \ \frac{1}{v(x)}\mathbf{E}_x[ R_n v(S_n); M_n < 0], \quad x \le 0 . 
\end{align*}
The probability measures  $\mathbf P_x^+$ and $\mathbf P_x^-$ correspond to conditioning  the random walk $S$  on not to enter $(-\infty,0)$ and $[0,\infty)$ respectively.\\ 
More precisely, if $R_n\rightarrow R$ a.s. with respect to $\mathbf{P}^+$ (resp. $\mathbf{P}^-$), then
\begin{align*}
\lim_{n\rightarrow\infty} &\mathbf{P}(R_n\in \cdot|L_n\geq 0)\rightarrow \mathbf{P}^+(R\in\cdot)\\ 
\lim_{n\rightarrow\infty} &\mathbf{P}(R_n\in \cdot|M_n<0)\rightarrow \mathbf{P}^-(R\in\cdot) .
\end{align*}
The first result is proved in \cite{kersting053}[Lemma 2.5] in the more general setting of random walks in the domain of attraction of a stable law. The proof of the second claim is analogous. 

We end up with an asymptotic result in the critical case, which is stated and proved only in the non-lattice case.
\begin{lemma}\label{le_critical}
We assume that $\mathbf{E}[X]=0$, $\mathbf{Var}(X)<\infty$, and that $\mathbf{E}\big[(\log^+ \xi_Q(a))^{2+\varepsilon}]<\infty$ for some $\varepsilon>0$. Then, for every $c>0$,
\begin{align*}
 \liminf_{n\rightarrow\infty} \mathbf{P}(Z_n>0|L_n\geq 0, S_n\leq c) >0 .
\end{align*}
\end{lemma}

\begin{proof} 
The proof follows essentially \cite{abkv10} and we just present the main steps. First,

 Lemmas \ref{le_as1} and \ref{le_as} ensure that  for all  $\theta, c>0$ large enough,  there exists  $d>0$ such that 
\begin{align}
 \mathbf{E}\big[e^{-\theta S_n}; L_n\geq 0\big] \sim d \ \mathbf{P}(L_n\geq 0, S_n\leq c)\qquad (n\rightarrow \infty). \label{help03}
\end{align}
Secondly, we recall the well-known estimate (see e.g. \cite{agresti}[Lemma 2]) 
\begin{align*}
 \mathbf{P}(Z_n>0 \ \vert \ \EE) \geq \frac{1}{e^{-S_n}+\sum_{i=0}^{n-1}\eta_{i+1} e^{-S_i} }\quad \text{a.s.},
\end{align*}
where $\eta_i:=\sum_{y=1}^\infty y(y-1)Q_i(y)/m_{Q_i}^2$. Then, we rewrite
\begin{align*}
\mathbf{E}\Big[&\frac{1}{e^{-S_n}+\sum_{i=0}^{n-1}\eta_{i+1} e^{-S_i} }; L_n\geq 0, S_n\leq c\Big]\\
&\geq \mathbf{E}\Big[\frac{1}{1+\sum_{i=0}^{\lfloor n/2\rfloor}\eta_{i+1} e^{-S_i}+e^{-S_{\lfloor n/2\rfloor}} 
\sum_{i=\lfloor n/2\rfloor +1}^{n-1}\eta_{i+1} e^{S_{\lfloor n/2\rfloor}-S_i} }; L_n\geq 0, S_n\leq c\Big] \\
&\geq \mathbf{E}\Big[\frac{(c-S_n)^+\wedge 1}{1+\sum_{i=0}^{\lfloor n/2\rfloor}\eta_{i+1} e^{-S_i}+ 
\sum_{i=\lfloor n/2\rfloor +1}^{n-1}\eta_{i+1} e^{S_{\lfloor n/2\rfloor}-S_i} }; L_n\geq 0\Big] \\
&= \mathbf{E}\Big[\varphi(U_n, \tilde V_n, S_n); L_n\geq 0\Big]\\
&\geq e^{-c/2} \mathbf{E}\Big[e^{-S_n/2}\varphi(U_n, \tilde V_n, S_n); L_n\geq 0\Big],
\end{align*}
where $U_n:= \sum_{i=0}^{\lfloor n/2\rfloor}\eta_{i+1} e^{-S_i}$, $\tilde V_n:=\sum_{i=\lfloor n/2\rfloor +1}^{n-1}\eta_{i+1} e^{S_{\lfloor n/2\rfloor}-S_i}$ and  $\varphi(u,v,z)=(1+u+v)^{-1} (c-z)^+\wedge 1$. 
Using (\ref{help03}), it becomes (with $\theta=\tfrac 12$)
\begin{align*}
\liminf_{n\rightarrow\infty} \mathbf{P}&(Z_n>0|L_n\geq 0, S_n\leq c) \geq d^{-1} \liminf_{n\rightarrow\infty} \frac{e^{-c/2} \mathbf{E}\Big[e^{-S_n/2}\varphi(U_n, \tilde V_n, S_n); L_n\geq 0\Big]}{\mathbf{E}\big[e^{-S_n/2}; L_n\geq 0\big]}.
\end{align*}
Due to monotonicity and Lemma 3.1 in \cite{abkv10}, 
the limits of $U_{\infty}=\lim_{n\rightarrow\infty}U_n$ and \linebreak $V_\infty=\lim_{n\rightarrow\infty}\sum_{i=0}^{\lfloor n/2\rfloor}\eta_{i} e^{S_i}$ exist 
and are finite respectively under the probabilities $\mathbf{P}^+$-a.s. and $\mathbf{P}^-$-a.s. 
Thus all conditions of Proposition 2.5 in \cite{abkv10} are met. Applying this proposition with $\theta=1/2$, 
there exists a non zero measure $\nu_{1/2}$ on $\mathbb{R}^+$ which gives the convergence of the right-hand side above and 
\begin{align*}
\liminf_{n\rightarrow\infty} \mathbf{P}&(Z_n>0|L_n\geq 0, S_n\leq c) & \geq \int_{\R_+^3} \varphi(u,v,-z) \mathbf{P}^+(U_\infty\in du)\mathbf{P}^-(V_\infty\in dv) \nu_{1/2}(dz) >0 .
\end{align*}
Note that in the function $\varphi$, $z$ is changed to $-z$ for duality reasons (see \cite{abkv10} for details). 
As  $U_{\infty}$ and $V_{\infty}$ are a.s. finite with respect to the corresponding measures, this yields the claim.
\end{proof}
\textbf{Remark.} The proof may be adapted to the lattice case, by proving for example that
\[\liminf_{n\rightarrow\infty}\mathbf{P}^+\Big(\sum_{i=0}^{n/2} \eta_{i+1}e^{-S_i} <d,S_n/\sqrt{n} \in (a,b) \Big) >0.\]
We  note that 
$\mathbf{P}^+(\sum_{i=0}^\infty \eta_{i+1}e^{-S_i}<\infty)=\mathbf{P}^-(\sum_{i=1}^\infty \eta_{i}e^{S_i}<\infty)=1$
has been proved in \cite{kersting053} also for the non-lattice case. 
But the main remaining problem is that Proposition 2.5 in \cite{abkv10} is only stated for 
non-lattice random walks. The generalization of this result is a technically involved task and beyond the scope of this paper.

\subsection{Proof of Proposition \ref{prop2} (i) : $\rho>0$}\label{rho0}
Under Assumption \ref{technical2} , we now prove that $\rho>0$. It means that  the
probability of staying small but alive is exponentially small. The proof relies again on the Geiger construction and results of the previous section. \medskip\\
We assume that there exists  $\gamma>0$ such that  $Q(0)<1-\gamma$ a.s. and $\mathbb{E}[|X|]<\infty$. Let $z_0$ be the smallest element in $\mathcal{I}$. Using (\ref{interpretation2}) and $(\ref{equat})$, we get 
\begin{align*}
 \lim_{n\rightarrow\infty}\tfrac 1n\log\mathbb{P}_{z_0}(Z_n=z_0)
 &= \lim_{n\rightarrow\infty}\tfrac 1n\log\mathbb{E}\big[\mathbb{P}_{z_0}(\hat{Z}^{(0)}_n+\ldots+\hat{Z}^{(n-1)}_n=0,\hat{Y}_n=z_0-1|\mathcal{E})\mathbb{P}_{z_0}(Z_n>0|\mathcal{E})\big]\\
&\leq \liminf_{n\rightarrow\infty}\tfrac 1n\log \mathbb{E}\Big[\prod_{j=0}^{n-1} \mathbb{P}(\hat{Z}^{(j)}_n=0|\mathcal{E})\Big] \\
& =\liminf_{n\rightarrow\infty}\tfrac 1n\log \mathbb{E}\Big[\exp\Big(\sum_{j=0}^{n-1}\log \mathbb{P}(\hat{Z}^{(j)}_n=0|\mathcal{E})\Big)\Big] .\label{10012012}
\end{align*} 
The fact that $\log(x)\leq x-1$ yields
\begin{align*}
\mathbb{E}\Big[\exp\Big(\sum_{j=0}^{n-1}\log \mathbb{P}(\hat{Z}^{(j)}_n=0|\mathcal{E})\Big)\Big]&\leq \mathbb{E}\Big[\exp\Big(-\sum_{j=0}^{n-1}\mathbb{P}(\hat{Z}^{(j)}_n>0|\mathcal{E})\Big)\Big] .
\end{align*} 
It remains to prove that this last expectation  decreases  exponentially. From  (\ref{equachap}), we get 
\begin{equation}\label{eq1234}
\mathbb{P}(\hat{Z}^{(j)}_n=0|\mathcal{E})=\frac{p_{j,n}}{p_{j-1,n}} f'_j(f_{j,n}(0))= \frac{1-f_{j,n}(0)}{1-f_j(f_{j,n}(0))} f'_j(f_{j,n}(0))=  h_j\big(f_{j,n}(0)\big),  
\end{equation}
where for $s \in [0,1)$,
$$h(s):=\frac{f'(s)}{g(s)}, \qquad  g(s):=\frac{1-f(s)}{1-s}.$$
We will now show that $g(1)=f'(1)$, $h(1)=1$. As already noticed in  \cite{BK09},  for every $s\in[0,1)$
$$g(s)=\sum_{k=0}^\infty \frac{1- s^k}{1-s}\frac{f^{(k)}(0)}{k!}= \sum_{k=1}^\infty (1+s+s^2+\ldots + s^{k-1}) \frac{f^{(k)}(0)}{k!} .$$
Thus $f'(1)=g(1)=1$ and $h(1)=1$. Moreover,   $f'(0)\ne 1$ ensures that for every $k>1$ and $s<1$, $ks^{k-1}<(1+s+s^2+\ldots + s^{k-1})$, so
 $h(s)<1$. A straightforward calculation shows that $h$ has exactly one minimum in some $s_0\in (0,1)$. Adding that $h$ is increasing for $s>s_0$ and $h(s_0)\leq h(0)$, we have every $t\in (0,1)$ that
\begin{align}\label{hmax} h(s)\leq \max\big\{h(0),h(t)\big\} \text{ for } s\leq t,\end{align}
First, we deal with $f_{j,n}(0)=\mathbb{P}(Z_n=0|\mathcal{E},Z_j=1)$ in (\ref{eq1234}). For this purpose, we use a truncation argument. Let $a\in\mathbb{N}$ be fixed for the moment and introduce 
\[ \bar{Q}(j):=Q(j), 1\leq j< a, \ \bar{Q}(a)=Q([a,\infty)). \]
We refrain from indicating the dependence on $a$ in our notation. The corresponding truncated random variables are denoted similarly, e.g. by $\bar{X}$, $\bar{S}$, $\bar{f}$. Note that $\bar{f}''(1)\leq a^2$.
By dominated convergence, we get that
\[\lim_{a\rightarrow\infty} \mathbb{E}[\bar{X}]=\mathbb{E}[X]>0.\]
Thus if $a$ is chosen large enough, $\bar{S}$ is still a random walk with positive drift and $\mathbb{E}[\bar{f}'(1)]>1$. Also note that with respect to the truncated offspring distributions, $\bar{Z}_n$ is stochastically 
smaller than $Z_n$ and thus $\mathbb{P}(\bar{Z}_n=0)\geq \mathbb{P}(Z_n=0)$. Applying this together with a well-known formula for the 
extinction probability (see e.g. \cite{agresti}[Lemma 2]), we get that
\begin{align*}
\mathbb{P}(Z_n=0|\mathcal{E},Z_j=1)\leq \mathbb{P}_1(\bar{Z}_n=0|\mathcal{E})&\leq 1- \frac{1}{e^{-\bar{S}_n}+\sum_{k=j}^{n-1} \bar{\eta}_{k+1} e^{-(\bar{S}_k-\bar{S}_j)}}
\\
&\leq 1-\frac{1}{a^2 \sum_{k=j}^\infty e^{-2\bar{X}_{k+1}-(\bar{S}_{k}-\bar{S}_j)}}, 
\end{align*}
where we used that $\bar{\eta}= \bar{f}''(1)/\bar{f}'(1)^2\leq a^2e^{-2\bar{X}}$ a.s. We now aim at bounding $\sum_{k=j}^\infty \exp(-2\bar{X}_{k+1}-(\bar{S}_{k}-\bar{S}_j))$. 
First, the assumption $Q(0)<1-\gamma$ implies that $\bar{X}\geq \log(\gamma)$ a.s. Thus
\begin{align*}
\mathbb{P}(Z_n=0|\mathcal{E},Z_j=1)&\leq \mathbb{P}_1(\bar{Z}_n=0|\mathcal{E})\leq 1-\frac{\gamma^2}{a^2 \sum_{k=j}^\infty e^{-(\bar{S}_{k}-\bar{S}_j)}}. 
\end{align*}
Next, we introduce the random walk $\breve{S}_n:=\bar{S}_n-\varepsilon n$ with $0<\varepsilon<\mathbb{E}[\bar{X}]$. 
It is still a random walk with positive drift and we have 
\begin{align*}
 f_{j,n}(0)&\leq 1-\frac{\gamma^2}{a^2\sum_{k=j}^\infty e^{-(\breve{S}_{k}-\breve{S}_j)}e^{-(k-j)\varepsilon}}. 
\end{align*}
Let us now consider  the prospective minima (see e.g. \cite{kersting053}[p.661]) of $\breve{S}$ which are defined by $\nu(0):=0$ and
\[\nu(j):=\inf\{n>\nu(j-1)  \ : \  \breve{S}_{k}>\breve{S}_n\ \forall k>n\}. \]
Then we can estimate for $j\geq 1$ (note that $\breve{S}_k\geq \breve{S}_{\nu(j)}$ for all $k\geq \nu(j)$, $j\geq 1$)
\begin{align*}
 f_{\nu(j),n}(0)&\leq 1-\frac{\gamma^2}{a^2\sum_{k=\nu(j)}^\infty e^{-(\breve{S}_{k}-\breve{S}_\nu(j))}e^{-(k-\nu(j))\varepsilon}} \\
 &\leq  1-\frac{\gamma^2}{a^2 \sum_{k=\nu(j)}^\infty e^{-(k-\nu(j))\varepsilon}} =1- \frac{\gamma^2(1-e^{-\varepsilon})}{a^2} .
\end{align*}
From  (\ref{hmax}), setting $d:=\gamma^2(1-e^{-\varepsilon})/a^2\in(0,1)$,  we get  for $j\geq 1$, 
\begin{align*}
 \mathbb{P}(\hat{Z}^{(\nu(j))}_n>0|\mathcal{E})= 1-h_{\nu(j)}\big(f_{\nu(j),n}(0)\big)& \geq   1-\max\{h(0),h_{\nu(j)}\big(1- d\big)\}\\
 &= \min\{1-h(0),1-h_{\nu(j)}\big(1- d\big)\}=:A_{\nu(j)} ,
\end{align*}
From the classical random walk theory, $U_j:=\nu(j)-\nu(j-1)$ 
(and also $Q_{\nu(j)}$) are i.i.d. random variables (see \cite{kersting053}).
We  now prove that for $\delta>0$ small enough that the probability
 that there are less than $\{\delta n\}$-many prospective minima is exponentially small. Note that $\mathbb{E}[\breve{X}]>0$ implies $\mathbb{E}[\nu(1)]<\infty$.
 Let $0<\delta<\mathbb{E}[\nu(1)]^{-1}$. Then
\begin{align}
\label{ctrlneg}
\mathbb{P}(\sharp\{j \geq 0 \ : \nu(j)\leq n\}<\delta n)\leq \mathbb{P}(\nu(\lceil \delta n\rceil)> n) \leq\mathbb{P}\Big(\sum_{j=1}^{\lfloor \delta n\rfloor}U_j >\frac{1}{\delta} n\delta\Big)\leq e^{-\delta n \Psi(\delta^{-1})},   
\end{align}
where $\Psi$ is the rate function of the process $(\sum_{j=1}^{ n}U_j)_n$, which is a random walk with nonnegative increments. 
Thus it remains to prove that $\Psi(\theta)>0$ for some $\theta>\mathbb{E}[U_1]=\mathbb{E}[\nu(1)]$. From large deviations theory,
 we just need to check that the tail of $\nu(1)$ decreases exponentially. This follows from 
\begin{align*}
 \mathbb{P}(\nu(1)>k)&\leq \mathbb{P}(S_j\leq 0 \ for \ some \ j>k) \leq \sum_{j=k}^\infty \mathbb{P}(S_j\leq 0) \\
&\leq \sum_{j=k}^\infty e^{-\breve{\Lambda}(0) j} = \ \frac{e^{-\breve{\Lambda}(0)k} }{1-e^{-\breve{\Lambda}(0)}}, 
\end{align*}
where $\breve{\Lambda}$ is the rate function of $\breve S$. This  rate function is proper since   $\log(1-\gamma)\leq \breve X\leq a$ a.s. Adding that  $\mathbb{E}[\breve X]>0$ ensures that
$\breve{\Lambda}(0)>0$ and we can conclude that $\Psi(\theta)>0$. \\
Finally, we use that $A_{\nu(j)}$ are independent to get 
 \begin{align*}
\mathbb{P}_{z_0}(Z_n=z_0)&\leq \mathbb{E}\Big[\exp\Big(-\sum_{j=0}^{n-1}\mathbb{P}(\hat{Z}^{(j)}_n>0|\mathcal{E})\Big)\Big] \\
 &\leq \mathbb{P}(\sharp\{j \geq 0 \ :  \nu(j)\leq n\}<\delta n)+ \mathbb{E}\Big[\exp\Big(-\sum_{j=0}^{\lfloor\delta n\rfloor}A_{\nu(j)}\Big)\Big] \\
 &\leq \mathbb{P}(\sharp\{j \geq  0 \ : \nu(j)\leq n\}<\delta n)+ \mathbb{E}\Big[\exp\Big(-A_{\nu(1)}\Big)\Big]^{\lfloor\delta n\rfloor}.
 \end{align*}
Recalling   that $A_{\nu(j)} \geq 0$ and $\P(A_{\nu(j)} > 0)=\P(h(0)<1)=\P(f'(0) \ne 1)>0$, we get  
$$\mathbb{E}\Big[\exp\Big(-A_{\nu(j)}\Big)\Big]<1.$$
Then (\ref{ctrlneg}) ensures that  $\rho>0$.\qed\medskip\\

\begin{Rk} To get Proposition \ref{prop2}  (i), Assumption $\ref{technical2}$ can be replaced by assuming that there exists 
$c$ such that $\eta\leq c$ a.s. The proof is very similar. In this case, the truncation is not required and we may estimate
$$\mathbb{P}(Z_n=0|\mathcal{E},Z_j=1)\leq \bar{\mathbb{P}}(Z_n=0)\leq 1- \frac{1}{e^{-S_n}+\sum_{k=j}^{n-1} \eta_{k+1} e^{-(S_k-S_j)}}\leq 1-\frac{1}{c \sum_{k=j}^\infty e^{-(S_{k}-S_j)}}.$$
\end{Rk}

\subsection{Proof of  Proposition \ref{prop2} (ii) : $\varrho\leq \Lambda (0)$}\label{section6}
Here, we prove the second part of Proposition \ref{prop2} which  ensures that $\varrho \leq \Lambda(0)$. It means 
 that small but positive values can always be realized by a suitable exceptional environment, which is 'critical'.
 We focus on the nontrivial case when $\Lambda(0)<\infty$. The proof of Proposition \ref{prop2} (ii) can then be splitted into two subcases, which correspond to the two following propositions.
\begin{prop}\label{prop0}
Under Assumption  \ref{axi} and $\P(X<0)>0$, we have
$\rho \leq \Lambda(0).$
\end{prop}

\begin{prop} \label{partII}
 Assume that $\mathbb{P}(X\geq 0)=1$ and $\mathbb{P}(X=0)>0$. Then
\begin{align}
 \rho\leq -\log \mathbb{P}(X=0)=\Lambda(0).
\end{align}
\end{prop}

\begin{proof}[Proof of Proposition \ref{prop0}]
We recall that $\mathcal{I}:=\big\{ j \geq 1 \ :\ \P(Q(j)>0, Q(0)>0)>0\big\}$. We use a standard approximation argument 
 and consider the event $E_{x,n}:=\{\min_{i=1,\ldots, n} X_i>x \}$ for $x<0$. Then, $\P(X>x)>0$ since we are in the supercritical regime and 
for every $s\geq 0$, $\E\big[\vert X \vert e^{-sX} |X>x\big]<\infty$. As $\P(X<0)>0$, we may choose $x$ small enough such that $\mathbb{P}(x<X<0)>0$. Then
 $\E\big[\vert X \vert e^{-sX} |X>x\big]$ tends to infinity as $s\rightarrow \infty$.  
Moreover $\mathbb{E}[e^{-sX}|X>x]$ is differentiable with respect to $s$ for $s>0$. We call 
$s=\nu_x$ a point where the minimum  of this function is reached. In particular, 
$$\inf_{s\geq 0}  \mathbb{E}[e^{-sX}|X>x] =  \mathbb{E}[e^{-\nu_x X}|X>-x], \qquad \tfrac{d}{ds}\mathbb{E}[e^{-sX}|X>x]\big|_{s=\nu_x}=\E[Xe^{-\nu_xX}|X>x]=0.$$
The second part of Lemma \ref{rho_1} ensures that for every $z\in \mathcal{I}$ and  for every sequence $k_n=o(n)$, 
\begin{align}
 \lim_{n\rightarrow\infty} \tfrac{1}{n} \log\mathbb{P}_{z}(Z_{n}=z)
=\lim_{n\rightarrow\infty} \tfrac{1}{n} \log\mathbb{P}_{z}(1\leq Z_{ n}\leq k_n)= -\rho. \nonumber
\end{align}
Let us now change to the measure $\mathbf{P}$, defined by
\begin{align}
 \mathbf{P}(X\in dy) = \frac{e^{-\nu_x y} \mathbb{P}(X\in dy|X>x)}{\mu}\label{bfp}
\end{align}
where $\mu:=\mathbb{E}[e^{-\nu_x X}|X>x]$. Under $\mathbf{P}$, $\mathbf{E}[X]=0$ and $S$ is a recurrent random walk. \\
Let $c>0$ be so large such that $\mathbf{P}(L_n\geq 0,S_n\leq c)>0$ for every $n$. Then
\begin{align}
\mathbb{P}_{z}(1\leq Z_{n}\leq k_n|E_{x,n}) &= \mu^{n}\mathbf{E}\Big[\mathbb{P}_{z}(1\leq Z_{n}\leq k_n|
\mathcal{E}) e^{\nu_x S_{n}}\Big]\nonumber \\
&\geq \mu^{n}\mathbf{E}\Big[\mathbb{P}_{z}(1\leq Z_{n}\leq k_n|\mathcal{E});
 L_{n}\geq 0, S_{n}\leq c\Big].  \label{he2}
\end{align}
We note that $ \mathbb{P}_{z}(1\leq Z_{n}\leq k_n|\mathcal{E})= \mathbb{P}_{z}(Z_{n}>0|\mathcal{E})- \mathbb{P}_{z}(Z_{n}> k_n|\mathcal{E}) \ \ \text{a.s.}$ and by Markov inequality,
$ \mathbb{P}_{z}( Z_{n}> k|\mathcal{E})\leq \frac{ze^{S_{n}}}{k} \ \  \text{a.s.}$
It ensures that  
\begin{align*}
 \mathbb{P}_{z}(1\leq Z_{n}\leq k_n|E_{x,n})\geq \mu^{n} \mathbf{E} \Big[\mathbf{P}(Z_n>0|\mathcal{E})-ze^c/k_n; L_{n}\geq 0, S_{n}\leq c\Big] .
\end{align*}
Plugging this into (\ref{he2}) and setting $b_n:=\mathbf{P}(L_n\geq 0, S_n\leq c)$, we get
\begin{align}
 \mathbb{P}_{z}(1\leq Z_{n}\leq k_n)& \geq \mathbb{P}_z(1\leq Z_n\leq k_n \vert  E_{x,n})\mathbb{P}(E_{x,n}) \nonumber \\
&= \ \mu^{n}
b_n \Big[\mathbf{P}(Z_n>0|L_{n}\geq 0, S_{n}\leq c)-ze^c/k_n\Big] \mathbb{P}(X>x)^n. \label{he3}
\end{align}
By construction of $\mathbf{E}$, $\mathbf{Var}(X)\leq \mu^{-1}\mathbb{E}[X^2 e^{-\nu_x X}|X>x]<\infty$. 
Then from Lemma \ref{le_as}, we have
$b_n=O(n^{-3/2})$ and 
$\lim_{n\rightarrow\infty} \tfrac{1}{n} \log b_n=0$. 
Let  $k_n=n^{-1/2}$. The fact that Assumption \ref{axi} holds under $\P$ entails that it holds also under $\mathbf{P}$.
Indeed,
\begin{align*}
\mathbb{E}\big[(\log^+ \xi_Q(a))^{2+\varepsilon}|X>x\big]&= \mu \mathbf{E}\big[e^{\nu_x X}(\log^+ \xi_Q(a))^{2+\varepsilon}\big]\\
&\geq \mu e^{\nu_x x}\mathbf{E}\big[(\log^+ \xi_Q(a))^{2+\varepsilon}\big],
\end{align*}
as $X>x$ ($x<0$) $\mathbf{P}$-a.s. Thus  we can use 
Lemma \ref{le_critical} and (\ref{he3}) to get 
\begin{align}
 \liminf_{n\rightarrow\infty} \tfrac{1}{n} \log \mathbb{P}_{z}(1\leq Z_{n}\leq k_n)&\geq \log \mu +\log \mathbb{P}(X>x)\nonumber\\
&=\log \mathbb{E}[e^{-\nu_x X}|X>x]+\log \mathbb{P}(X>x)\nonumber \\
&= -\sup_{s\leq 0}\big\{- \log \mathbb{E}[e^{-s X};X>x]\big\}.  \nonumber
\end{align}
By monotone convergence, we let  $x \rightarrow-\infty$ and
\begin{align}
 \liminf_{n\rightarrow\infty} \tfrac{1}{n} \log \mathbb{P}_{z}(1\leq Z_{n}\leq k_n)\geq -\sup_{s\leq 0}\big\{- \log \mathbb{E}[e^{-s X}]\big\}=-\Lambda(0). \nonumber
\end{align}
As $k_n=o(n)$, we apply Lemma $\ref{rho_1}$ to end up the proof.
\end{proof}

\begin{proof}[Proof of Proposition \ref{partII}]
As $\mathbb{P}(X\geq 0)=1$, we have $\mathbb{P}(S_n=0)=\mathbb{P}(X=0)^n$ and  $\Lambda(0)=-\log \mathbb{P}(X=0)$.\\
If $\mathbb{P}(Q(1)=1|X=0)=1$, the proof is trivial. So let us work with the assumptions  $\mathbb{P}(X=0)>0$ and $\mathbb{P}(Q(1)=1|X=0)<1$. 

 By conditioning on the environment, we get  for $z\in\mathcal{I}$ that
\begin{align}
 \mathbb{P}_{z}(Z_n=z) \geq \mathbb{P}(X=0)^n \cdot\mathbb{P}_{z}(Z_n=z|X_1=0,\ldots, X_n=0).\nonumber
\end{align}
For simplicity, we introduce a new measure $\bar{\mathbf{P}}$ on the space of all probability measures on $\mathbb{N}_0$ with expectation $1$. It is defined for 
 every measurable $A\subset \Delta$ by
\begin{align*}
\bar{\mathbf{P}}(Q\in A):= \frac{\mathbb{P}(Q\in A; m_Q=1)}{\mathbb{P}(m_Q=1)} = \frac{\mathbb{P}(Q\in A;m_Q=1)}{\mathbb{P}(X=0)}  .
\end{align*}
Note that $\bar{\mathbf{P}}(X=0)=1$ and there exists  $z\geq 1$ such that $\bar{\mathbf{P}}\big(Q(z)>0,Q(0)>0\big)>0$. 
With respect to $\bar{\mathbf{P}}$, $(Z_n\ :\ n\in\N_0)$ is still a branching process in random environment. By  Lemma \ref{rho_1}, there exists  $\bar{\rho} \in [0,\infty)$ such that
\begin{align}
\bar{\rho}&= -\lim_{n\rightarrow\infty} \tfrac{1}{n} \log \bar{\mathbf{P}}_z (Z_n=z)\nonumber \\
&=-\lim_{n\rightarrow\infty} \tfrac{1}{n} \log \mathbb{P}_{z}(Z_n=z|X_1=0,\ldots, X_n=0)\nonumber \\
&= -\lim_{n\rightarrow\infty} \tfrac{1}{n} 
\log \mathbb{E}\Big[Q_n(z)f_{0,n}^{z-1}(0)\prod_{i=1}^{n-1} f'_i\big(f_{i,n}(0)\big)\Big|X_1=0,\ldots, X_{n}=0\Big]. \nonumber
\end{align}
Next, we use convexity arguments. 
First,  for all $i\leq k$ and $s\in[0,1]$, $f_{i,k}(s)\geq 1-f'_{i,k}(1)(1-s)$. As $\bar{\mathbf{P}}(f_{i,k}'(1)=1)=1$, we get 
\begin{align}
 f_{i,k}(s)\geq s \quad  \bar{\mathbf{P}}\text{-a.s.}\label{pbarhelp}
\end{align}
Also recall that $f_{0,n}(0)=f_{0,n-1}(f_n(0))$ and by (\ref{pbarhelp}), $f_{i,n}(0)\geq f_n(0)=Q_n(0)$. Thus,  for every $a\in\mathbb{N}$, 
\begin{align}
 \liminf_{n\rightarrow\infty}& \tfrac{1}{n}\log\mathbb{E}\Big[Q_n(z)f_{0,n}^{z-1}(0)\prod_{i=1}^{n-1} f'_i\big(f_{i,n}(0)\big)\Big|X_1=0,\ldots, X_{n}=0\Big]
\nonumber \\
&\geq \liminf_{n\rightarrow\infty} \tfrac{1}{n}\log \bar{\mathbf{E}}\Big[Q_n(z)Q_n(0)^{z-1}\prod_{i=1}^{n-a} f'_i\big(f_{n-a,n}(0)\big)
\prod_{i=n-a+1}^{n-1} f'_i\big(f_{i,n}(0)\big)\Big].\nonumber
\end{align}
For $\varepsilon>0$  fixed, we choose $k=k_\varepsilon\in\mathbb{N}$ large enough such that $\bar{\mathbf{P}}(Q([1,k])>\varepsilon)\geq 1-\varepsilon$. Then, conditionally on $\{Q([1,k])>\varepsilon\}$,  $f'(s)\geq \sum_{j=1}^k Q(k) s^k\geq \varepsilon s^k$ a.s. for $s\in[0,1]$. Using this inequality, (\ref{pbarhelp}) and $f_{i,n}(0)\leq f_{i,n-1}(0)$, we have 
\begin{align*}
\bar{\mathbf{E}}&\Big[Q_n(z)Q_n(0)^{z-1}\prod_{i=1}^{n-a} f'_i\big(f_{n-a,n}(0)\big)\prod_{i=n-a+1}^{n-1} f'_i\big(f_{i,n}(0)\big)\Big| Q_{1},\ldots, Q_{n-a}\Big] \nonumber \\
&\geq \bar{\mathbf{E}}\Big[Q_n(z)Q_n(0)^{z-1}\prod_{i=1}^{n-a} f'_i\big(f_{n-a,n}(0)\big)\\
&\qquad \qquad \times \prod_{i=n-a+1}^{n-1} f'_i\big(Q_n(0)\big); Q_{n-a+1}([1,k])>\varepsilon,\ldots,Q_{n-1}([1,k])>\varepsilon \Big| Q_{1},\ldots, Q_{n-a}\Big]\\
 &\geq \bar{\mathbf{E}}\Big[Q_n(z)Q_n(0)^{z-1}\prod_{i=1}^{n-a} f'_i\big(f_{n-a,n-1}(0)\big)\\
&\qquad \qquad \times \prod_{i=n-a+1}^{n-1} \varepsilon Q_n(0)^k; Q_{n-a+1}([1,k])>\varepsilon,\ldots,Q_{n-1}([1,k])>\varepsilon\Big| Q_{1},\ldots, Q_{n-a}\Big]\\
 &\geq \bar{\mathbf{E}}\Big[\prod_{i=1}^{n-a} f'_i\big(f_{n-a,n-1}(0)\big); Q_{n-a+1}([1,k])>\varepsilon,\ldots,Q_{n-1}([1,k])>\varepsilon\Big| Q_{1},\ldots, Q_{n-a}\Big]\\
 &\qquad \qquad \times \bar{\mathbf{E}}\Big[\varepsilon^{a-2} Q_n(z)Q_n(0)^{z-1+(a-2)k}\Big],\nonumber 
\end{align*}
where the second expectation is strictly positive as $\bar{\mathbf{P}}(Q(z)>0,Q(0)>0)>0$.  The product of two generating functions (and thus the product of finitely many) is again convex. Indeed
 generating functions, as well as all their derivatives are convex, nonnegative and nondecreasing functions, thus
\[ (fg)''=f''g+2g'f'+fg'' \geq 0. \]
Similarly,  the product of the derivatives of generating functions is again convex. For more details on the product of nonnegative, convex and nondecreasing functions, we refer to \cite{marchi}. Applying Jensen's inequality to the convex function $\Pi_{i=1}^{n-a} f_i'$, the independence of the environments ensures that 
\begin{align*}
\bar{\mathbf{E}}\Big[&\prod_{i=1}^{n-a} f'_i\big(f_{n-a,n-1}(0)\big);Q_{n-a+1}([1,k])>\varepsilon,\ldots,Q_{n-1}([1,k])>\varepsilon\Big| Q_{1},\ldots, Q_{n-a}\Big]\\
& \geq \prod_{i=1}^{n-a} 
f'_i\big(\bar{\mathbf{E}}\big[f_{n-a,n-1}(0);Q_{n-a+1}([1,k])>\varepsilon,\ldots,Q_{n-1}([1,k])>\varepsilon\big| Q_1,\ldots,Q_{n-a}\big]\big) \\
&\qquad \qquad = \prod_{i=1}^{n-a} f'_i\big(\bar{\mathbf{E}}\big[f_{0,a-1}(0);Q_{1}([1,k])>\varepsilon,\ldots,Q_{a-1}([1,k])>\varepsilon\big]\big) \qquad \bar{\mathbf{P}}\text{-a.s.}
\end{align*}
 Using this inequality yields
\begin{align}
& \liminf_{n\rightarrow\infty} \tfrac{1}{n}\log\mathbb{E}\Big[f_{0,n}^{z-1}(0)\prod_{i=1}^{n} f'_i\big(f_{i,n}(0)\big)\Big|X_1=0,\ldots, X_n=0\Big]\nonumber \\
&\qquad \geq \liminf_{n\rightarrow\infty} \tfrac{1}{n}\log\Bigg(\bar{\mathbf{E}}\Big[\prod_{i=1}^{n-a} f'_i\big(\bar{\mathbf{E}}\big[f_{0,a-1}(0);Q_{1}([1,k])>\varepsilon,\ldots,Q_{a-1}([1,k])>\varepsilon\big]\big)\Big]  \nonumber\\
&\qquad \qquad \qquad \qquad \qquad \qquad \qquad 
\times \bar{\mathbf{E}}\Big[\varepsilon^{a-2} Q_n(z)Q_n(0)^{z-1}(0) Q_n(0)^{(a-2)k}\Big]\Bigg)\nonumber \\
&\qquad = \liminf_{n\rightarrow\infty} \tfrac{1}{n}\log\bar{\mathbf{E}}\Big[f'\big(\bar{\mathbf{E}}\big[f_{0,a-1}(0);Q_{1}([1,k])>\varepsilon,\ldots,Q_{a-1}([1,k])>\varepsilon\big]\big)\Big]^{n-a}\nonumber \\
&\qquad = \log\bar{\mathbf{E}}\Big[f'\big(\bar{\mathbf{E}}\big[f_{0,a-1}(0);Q_{1}([1,k])>\varepsilon,\ldots,Q_{a-1}([1,k])>\varepsilon\big]\big)\Big]. \nonumber
\end{align}
Finally, $Z$ is a critical branching process in random environment under the probability $\bar{\mathbf{P}}$
so  $\bar{\mathbf{P}}(Z_{a-1}=0|\mathcal{E}) =f_{0,a-1}(0)\rightarrow 1$ $\bar{\mathbf{P}}$-a.s. 
as $a\rightarrow\infty$ (see e.g. \cite{smith69}). Letting  $a\rightarrow\infty$,  $\varepsilon\rightarrow 0$ and recalling that $\bar{\mathbf{P}}(Q([1,k])>\varepsilon)\geq 1-\varepsilon$ yields  by dominated convergence 
\begin{align*} \log\bar{\mathbf{E}}\Big[f'\big(\bar{\mathbf{E}}\big[f_{0,a-1}(0);Q_{1}([1,k])>\varepsilon,\ldots,Q_{a-1}([1,k])>\varepsilon\big]\big)\Big]\rightarrow \log\bar{\mathbf{E}}\big[f'(1)\big]=0. \nonumber
\end{align*}
Then,
\begin{align*}
 \liminf_{n\rightarrow\infty}& \tfrac{1}{n}\log\mathbb{E}\Big[f_{0,n}^{z-1}(0)\prod_{i=1}^{n-1} f'_i\big(f_{i,n}(0)\big)\Big|X_1=0,\ldots, X_n=0\Big]\geq 0 .
\end{align*}

This yields the claim.
\end{proof}
\begin{Rk}
 Note that the bound $f'(s)\leq f'(1)$ for $s\in[0,1]$ immediately yields that \linebreak
$\limsup_{n\rightarrow\infty} \tfrac{1}{n} \log \mathbb{P}_{z}(Z_n=z)\leq \log\mathbb{E}[X]$. In particular, we recover that for a BPRE with  $X=0$ a.s.,  
the probability of staying bounded but positive is not exponentially small. \\
\end{Rk}

\section{The linear fractional case : Proof of Corollary \ref{lf1} }
In this section, we assume that  the offspring distributions have  generating functions of linear fractional form, i.e.
\begin{eqnarray}
 f(s)\ =\ 1- \frac{1-s}{m^{-1}+ b\ m^{-2} (1-s)/2},  \nonumber
\end{eqnarray}
where $m=f'(1)$ and $b=f''(1)$. 
\medskip \\
Under this assumption, direct calculations with generating functions are feasible, i.e. we can explicitly calculate the generating function of $Z_n$, 
conditioned on the environment. We also assume that 
$\mathbb{E}[|X|]<\infty$,  $\mathbb{P}(Z_1=0)>0$ and that either  $\P(X\geq 0)=1$ or Assumption \ref{axi} hold, such that Proposition \ref{prop2} (ii) holds.\\
In the next subsection, we prove Corollary \ref{lf1}. It  gives an expression of $\varrho$ which depends on the sign of $\E[X\exp(-X)]$. Afterwards, we prove Corollary \ref{corMRCA} which concerns the MRCA. Let us define 
 $\eta_k:=  b_{k} m^{-2}_{k}/2$ and recall that $f_{j,n}=f_{j+1}\circ\ldots\circ f_n$. Then for all $n\in\mathbb{N}$ and $s\in[0,1]$ (see \cite[p. 156]{kozlov06})
\begin{align*}
 f_{0,n}(s)&=1-\frac{(1-s)}{e^{-S_n}+(1-s) \sum_{k=0}^{n-1} \eta_{k+1} e^{-S_{k}}}.
\end{align*}
resp.
\begin{align}
 f_{j,n}(0)&=1-\frac{1}{e^{-(S_n-S_{j})}+\sum_{k=j+1}^{n} \eta_k e^{-(S_{k-1}-S_j)}}.\label{genf1}
\end{align}
Let us state some direct consequences resulting from this formula which will be used later. Taking the derivative, 
\begin{align}
 f_{0,n}'(s)&=\frac{e^{-S_n}}{\big(e^{-S_n}+(1-s)\sum_{k=0}^{n-1} \eta_{k+1} e^{-S_{k}}\big)^2}.\label{genf20}
\end{align}
Note that for every $s\in [0,1)$, applying (\ref{genf1})
\begin{align}
f_{0,n}'(s) (1-s)^2&= \frac{e^{-S_n}}{\big((1-s)^{-1}e^{-S_n}+\sum_{k=0}^{n-1} \eta_{k+1} e^{-S_{k}}\big)^2}\nonumber\\
&\leq e^{-S_n} \frac{1}{\big(e^{-S_n}+\sum_{k=0}^{n-1} \eta_{k+1} e^{-S_{k}}\big)^2}\nonumber\\
&= e^{-S_n} \big(1-f_{0,n}(0)\big)^2 = e^{-S_n} \mathbb{P}(Z_n>0|\mathcal{E})^2. \label{genf30}
\end{align}
Moreover,
\begin{align}
 f_j'(s)=\frac{e^{-X_j}}{(e^{-X_j}+\eta_j(1-s))^2} \label{genf2}
\end{align}
and we can now compute the value of $\varrho$.\\

\subsection{Determination of the value of $\varrho$}\label{linearfrac}
By Proposition \ref{prop2} (ii), $\rho \leq \Lambda(0)$. Then it remains to prove that $\rho=-\log\mathbb{E}\big[ e^{-X}\big]$ if $\mathbb{E}[Xe^{-X}]\geq 0$ and $\rho\geq \Lambda(0)$ 
otherwise.  For that purpose, we  use the representation
of $\rho$ in terms of generating functions. Combining (\ref{genf1})  and (\ref{genf2})
we get 
\begin{align}
 f_j'\big(f_{j,n}(0)\big)&=e^{-X_j}\big(e^{-X_j}+\tfrac{\eta_j}{e^{-(S_n-S_{j})}+\sum_{k=j+1}^{n} \eta_{k} e^{-(S_{k-1}-S_j)}}\big)^{-2} \nonumber\\
&=e^{-X_j} \Big(\frac{e^{-(S_n-S_{j})}+\sum_{k=j+1}^{n} \eta_k e^{-(S_{k-1}-S_j)}}{e^{-(S_n-S_{j-1})}+\sum_{k=j}^{n} \eta_{k} e^{-(S_{k-1}-S_{j-1})}}\Big)^2 \nonumber \\
&= e^{-X_j} \Big(\frac{\mathbb{P}(Z_n>0|Z_{j-1}=1,\mathcal{E})}{\mathbb{P}(Z_n>0|Z_{j}=1,\mathcal{E})}\Big)^2. \nonumber 
\end{align}
Since   $\mathbb{P}(Z_n>0|Z_n=1)=1$, we get 
\bea
\label{utile}
 \P_1(Z_n=1)&=&\E_1\Big[\prod_{j=1}^{n}f_j'\big(f_{j,n}(0)\big)\Big] \nonumber \\
&=&\E_1\Big[\prod_{j=1}^{n}e^{-X_j}\frac{\mathbb{P}(Z_n>0|Z_{j-1}=1,\mathcal{E})^2}{\mathbb{P}(Z_n>0|Z_{j}=1,\mathcal{E})^2}\Big] \nonumber \\
&=& e^{-S_{n}} \mathbb{P}(Z_n>0|\mathcal{E},Z_0=1)^2.
\eea
First, we consider the case $\mathbb{E}[Xe^{-X}]\geq 0$. Bounding the probability above  by $1$ immediately yields 
$$\mathbb{E}\Big[\prod_{j=1}^{n}f_j'\big(f_{j,n}(0)\big) \Big]\leq \mathbb{E}[e^{-S_n}]=\mathbb{E}[e^{-X}]^n.$$ 
Thus  $\rho\geq -\log \mathbb{E}[e^{-X}]$. To get the converse inequality,
 we  change to the measure $\hat{\mathbb{P}}$, defined by 
\begin{align}
 \hat{\mathbb{P}}(X\in dx) &= \frac{e^{-x}\mathbb{P}(X\in dx)}{\mathbb{E}[e^{-X}]}. \nonumber
\end{align}
This measure is well-defined as $\mathbb{E}\big[X^2e^{-X}\big]<\infty$ implies $\mathbb{E}\big[e^{-X}\big]<\infty$. Then by Jensen's inequality, 
\begin{align}
\mathbb{E}_1\Big[e^{-S_{n}} &\mathbb{P}(Z_n>0|\mathcal{E})^2\Big] = \mathbb{E}\big[e^{-X}\big]^{n} \hat{\mathbb{E}}\big[\mathbb{P}_1(Z_n>0|\mathcal{E})^2\big]\geq \mathbb{E}[e^{-X}]^{n} \hat{\mathbb{P}}_1(Z_n>0)^2. \nonumber 
\end{align}
We observe that $\hat{\mathbb{E}}[X]= \mathbb{E}[Xe^{-X}]\geq 0$, such that under $\hat{\P}$, $S_n$ is a random walk with nonnegative drift. It ensures that the branching process is still 
critical or supercritical with respect to $\hat{\mathbb{P}}$. Thus, under Assumption \ref{axi} and as $\hat{\mathbb{E}}[X]=\mathbb{E}[X^2e^{-X}]<\infty$,
$\hat{\mathbb{P}}(Z_n>0)>Cn^{-\tfrac 12}$ for some $C>0$ as $n\rightarrow\infty$ (see e.g. \cite{kersting053} for the critical case, whereas $\mathbb{P}(Z_n>0)$ has a positive limit in the supercritical case). Letting $n\rightarrow \infty$ and adding that $1 \in \mathcal{I}$, we get 
\begin{align*}
 \rho=-\lim_{n\rightarrow\infty}\tfrac 1n \log  \mathbb{E}\Big[\prod_{j=1}^{n}f_j'\big(f_{j,n}(0)\big) \Big]\leq -\log\mathbb{E}[e^{-X}] .
\end{align*}
Secondly, we consider $\mathbb{E}[Xe^{-X}]<0$. Then there exists   $\nu\in(0,1)$ such that $\mathbb{E}[Xe^{-\nu X}]=0$ and we 
change to the measure $\mathbf{P}$ defined in (\ref{bfp}) (here without truncation).
Applying this change of measure and the well-known estimate $\mathbb{P}(Z_n>0|\mathcal{E})\leq e^{L_n\wedge 0}$ a.s., we get that
\begin{align*}
 \mathbb{E}\Big[e^{-S_{n}} &\mathbb{P}(Z_n>0|\mathcal{E},Z_0=1)^2\Big]\leq \mathbb{E}[e^{-\nu X}]^{n}\mathbf{E}\Big[e^{(-1+\nu)S_{n}+2L_n\wedge 0}\Big].
\end{align*}
Note that $L_n\wedge 0\leq \min(S_n,0)$ and $\nu\in(0,1)$  that $(-1+\nu)S_{n}+2L_n\wedge 0\leq 0$, and thus 
 \begin{align*}
 \mathbb{E}\Big[e^{-S_{n}} &\mathbb{P}(Z_n>0|\mathcal{E},Z_0=1)^2\Big]\leq \mathbb{E}[e^{-\nu X}]^{n}.
\end{align*}
This  yields $\rho\geq -\log \mathbb{E}[e^{-\nu X}]=\Lambda(0)$ since   $\Lambda(0)=\sup_{s\leq 0} \{-\log\mathbb{E}[e^{sX}]\}$ and  the condition 
$\mathbb{E}[Xe^{-\nu X}]=0$ implies that the supremum is taken in $s=-\nu$.

\subsection{Proof of the limit theorems for the MRCA}\label{secmrca}
The three cases (i-ii-iii) in Corollary \ref{corMRCA} result respectively from Lemma \ref{le1}, Lemmas \ref{le2}, \ref{mrcaint2}, \ref{mrcaint3} and Lemma \ref{le3} below. In all this Section, we assume that the 
assumptions of Corollary \ref{lf1} are met, i.e.  that $\mathbb{E}[X^2e^{-X}]<\infty$, $\mathbb{E}[|X|]<\infty$ and either $\mathbb{P}(X\geq 0)=1$ or 
Assumption \ref{axi} holds. 
\begin{lemma}\label{le1}
If $\mathbb{E}[Xe^{-X}]<0$, then
\[\liminf_{n\rightarrow\infty} \mathbb{P}_1(MRCA_n=n|Z_n=2)\in (0,1)\  \]
and
\[\liminf_{n\rightarrow\infty} \mathbb{P}_1(MRCA_n=1|Z_n=2)\in (0,1). \]
\end{lemma}
\begin{proof} 
Conditionally on $\mathcal{E}$, the branching property holds and ensures that
\begin{align*}
 \mathbb{P}_1(MRCA_n=n|Z_n=2)&\geq \mathbb{P}_1(Z_1=2)\cdot\frac{\mathbb{E}\big[\mathbb{P}_1(Z_{n-1}=1|\mathcal{E})^2 \big]}{\mathbb{P}_1(Z_{n}=2)},
\end{align*}
which corresponds to two subtrees being founded by $Z_0=1$ and staying equal to 1 in all generations. \\
Next, we recall that linear fractional offspring distributions are stable with respect to compositions and  have geometrically decaying probability weights (see e.g. \cite{kozlov06}). In particular, 
\begin{align}
\label{relLF}
\mathbb{P}_1(Z_{n}=2|\mathcal{E})\leq\mathbb{P}_1(Z_{n}=1|\mathcal{E})
\end{align}
and therefore
\begin{align}
\label{qe1}
 \mathbb{P}_1(MRCA_n=n|Z_n=2)&\geq \mathbb{P}_1(Z_1=2)\cdot\frac{\mathbb{E}\big[\mathbb{P}_1(Z_{n-1}=1|\mathcal{E})^2 \big]}{\mathbb{P}_1(Z_{n}=1)}.
\end{align}
 Note that $\mathbb{E}[Xe^{-X}]<0$ implies that there exists  $\nu\in (0,1)$ such that $\mathbb{E}[Xe^{-\nu X}]=0$. Thus we can apply the same change of measure as in the proof of 
Proposition \ref{prop0}. With the definition of $\mathbf{P}$ therein (again without truncation), we get 
\begin{align}\label{eq_main1}
 \frac{\mathbb{E}[\mathbb{P}_1(Z_n=1|\mathcal{E})^2]}{\mathbb{P}_1(Z_n=1)}= \frac{\mathbf{E}[e^{\nu S_n}\mathbb{P}_1(Z_n=1|\mathcal{E})^2]}{\mathbf{E}[e^{\nu S_n}\mathbb{P}_1(Z_n=1|\mathcal{E})]} .
\end{align}
From (\ref{utile}), we know that
\begin{align}\label{eq_p1}
\mathbb{P}_1(Z_n=1|\mathcal{E})=f_{0,n}'(1)=e^{-S_n} \mathbb{P}_1(Z_n>0|\mathcal{E})^2 \leq e^{-S_n+2L_n} \ \text{a.s.} 
\end{align}
Combining this  with Jensen inequality yields for every $c>0$,
\begin{align}\label{mrca_11}
\mathbf{E}[e^{\nu S_n}\mathbb{P}_1(Z_n=1|\mathcal{E})^2]& \geq \mathbf{E}\big[e^{-2S_n+\nu S_n} \mathbb{P}_1(Z_n>0|\mathcal{E})^4 ; L_n\geq 0, S_n<c\big] \nonumber\\
&\geq e^{(-2+\nu)c} \mathbf{E}\big[\mathbb{P}_1(Z_n>0|\mathcal{E})^4; L_n\geq 0, S_n<c\big]\nonumber \\
&\geq e^{(-2+\nu)c}\mathbf{P}( L_n\geq 0, S_n<c) \mathbf{E}\big[\mathbb{P}_1(Z_n>0|\mathcal{E})^4| L_n\geq 0, S_n<c\big]\nonumber\\
& \geq e^{(-2+\nu)c}\mathbf{P}( L_n\geq 0, S_n<c) \mathbf{P}(Z_n>0| L_n\geq 0, S_n<c)^4 \nonumber\\
&\geq d\ \mathbf{P}( L_n\geq 0, S_n<c),
\end{align}
for some constant $d>0$, where the last line follows from Lemma \ref{le_critical}.
For the denominator in (\ref{eq_main1}), by (\ref{eq_p1}), we get similarly
\begin{align*}
 \mathbf{E}[e^{\nu S_n}\mathbb{P}_1(Z_n=1|\mathcal{E})]\leq \mathbf{E}[e^{-(1-\nu) S_n+2L_n}].
\end{align*}
Finally, we use $\E[X^2\exp(-X)]$ to ensure that $\mathbf{E}[X^2]<\infty$ and apply Lemmas \ref{le_as1} and \ref{le_as}. Then, 
\be
\liminf_{n\rightarrow\infty}   \frac{\mathbf{E}[e^{\nu S_n}\mathbb{P}_1(Z_n=1|\mathcal{E})^2]} {\mathbf{E}[e^{\nu S_n}\mathbb{P}_1(Z_n=1|\mathcal{E})] } \geq \liminf_{n\rightarrow\infty}  \frac{\mathbf{P}( L_n\geq 0, S_n<c)}{\mathbf{E}[e^{-(1-\nu) S_n+2L_n}]}>0.\nonumber
\ee
Recalling (\ref{qe1})   yields the first part of the lemma. \\
Similarly, we note that by the Markov  property,
 \begin{align*}
 \mathbb{P}_1(MRCA_n=1,Z_n=2)\geq \mathbb{P}_1(Z_{n-1}=1)\mathbb{P}(Z_1=2).
 \end{align*}
 Recalling that $\mathbb{P}_1(Z_{n-1}=1)\geq \mathbb{P}_1(Z_{n-1}=2)$ from $(\ref{relLF})$ yields the second claim. 
\end{proof}

The next three lemmas cover the  'intermediate' regime. First, we prove that the probability of $\{MRCA_n=x_n\}$ doesn't decay exponentially for every $1 \leq x_n \leq n$.
\begin{lemma}\label{le2}
If  $\mathbb{E}[Xe^{-X}]=0$, then for every $1 \leq x_n \leq n$
\[\lim_{n\rightarrow\infty}\tfrac 1n \log\mathbb{P}_1(MRCA_{n}=x_n|Z_{n}=2)=0. \]
\end{lemma}
\begin{proof} 
Let $\mathbb{E}[Xe^{-X}]=0$ and note that $MRCA_{n}\geq 1$.  Conditionally on $\mathcal{E}$, the branching property holds and guarantees that
 as in the previous proof that for every $1\leq x_n\leq n$, 
\begin{align*}
 \mathbb{P}_1(MRCA_{n}=x_n |Z_{n}=2)&\geq \mathbb{P}_1(Z_{n-x_n} =2)\cdot
\frac{\mathbb{E}\big[\mathbb{P}_1(Z_{x_n}=1|\mathcal{E})^2 \big]}{\mathbb{P}_1(Z_{n}=2)} .
\end{align*}
As $\E[X\exp(-X)]\geq 0$, we know from the previous subsection that $\log\mathbb{E}[e^{-X}]=\lim_{n\rightarrow\infty} \mathbb{P}_1(Z_n=2)$. Thus we get for $1 \leq x_n \leq n$, 
\begin{align*}
\liminf_{n\rightarrow\infty}&\tfrac 1n \log \mathbb{P}_1(MRCA_{n}=x_n |Z_{n}=2)\nonumber\\
&\geq \liminf_{n\rightarrow\infty} \bigg\{ (1-x_n/n) \log\mathbb{E}[e^{-X}] +    \tfrac 1n \log\mathbb{E}\big[\mathbb{P}_1(Z_{x_n}=1|\mathcal{E})^2 \big] \bigg\}  - \log\mathbb{E}[e^{-X}].\label{inter1*}
\end{align*}
For the last term, we use the  change of measure of the  previous lemma with $\nu=1$, i.e.
\[\mathbf{P}(X \in dx)=\frac{e^{-x} \mathbb{P}(X\in dx)}{\mathbb{E}[e^{-X}]}. \]
Applying (\ref{mrca_11}), we get that
\begin{align*}
 \mathbb{E}\big[\mathbb{P}_1(Z_{x_n}=1|\mathcal{E})^2 \big]\geq d\ \mathbb{E}[e^{-X}]^{x_n} \mathbf{P}( L_{x_n}\geq 0, S_{x_n}<c) .
\end{align*}
As  $\mathbb{E}[Xe^{-X}]=0$, $S$ is a recurrent random walk under $\mathbf{P}$. Using again Lemma \ref{le_as}, $\mathbf{P}( L_n\geq 0, S_n<c)\sim d n^{-\tfrac 32}$,
and 
\begin{align*}
\liminf_{n\rightarrow\infty}&\tfrac 1n \log \mathbb{P}_1(MRCA_{n}=x_n  |Z_{n}=2)\nonumber\\
&\geq \log\mathbb{E}[e^{-X}] - \log\mathbb{E}[e^{-X}]+ \liminf_{n\rightarrow\infty} \bigg\{ -x_n/n \log\mathbb{E}[e^{-X}] +   x_n/n\log\mathbb{E}[e^{-X}]\bigg\} = 0.
\end{align*}
It gives the expected lower bound,  whereas the upper bound follows is simply due to the fact that
the probabilities are less than $1$. 
\end{proof}
The following lemma is required to prove limit results for the MRCA in the intermediately supercritical case. To avoid technicalities, we impose an additional moment condition.
\begin{lemma}\label{helpint}
 Let $\mathbb{E}\big[Xe^{-X}\big]=0$ and $\gamma=\mathbb{E}[e^{-X}]$. Assume that $\mathbb{E}[f''(1)e^{-X}]<\infty$. Then
\[0<\liminf_{n\rightarrow\infty} \gamma^{-n} n^{1/2}\ \mathbb{P}_1(Z_{n}=2) \leq \limsup_{n\rightarrow\infty} \gamma^n n^{1/2}\ \mathbb{P}_1(Z_{n}=2) <\infty, \]
i.e. $\mathbb{P}_1(Z_{n}=2)$ is of the order $\gamma^nn^{-1/2}$.
\end{lemma}
\begin{proof}
Using the change of measure to $\mathbf{P}$ and (\ref{eq_p1}),
\begin{align}
\mathbb{P}_1(Z_n=2) &=\mathbb{E}\big[e^{-S_n} \mathbb{P}_1(Z_n>0|\mathcal{E})^2\big]  \nonumber \\
&= \gamma^n \mathbf{E}[\mathbb{P}(Z_n>0|\mathcal{E})^2] \geq \gamma^n \mathbf{P}(Z_n>0). \label{help23}
\end{align}
Under $\mathbf{P}$, $\mathbf{E}[X]=0$ and $Z$ is a critical branching process in random environment. Under our assumptions, 
there is a constant $d>0$ such that (see  \cite{kersting052}[Theorem 1.1]) 
\begin{align}\label{help24}
\mathbf{P}(Z_n>0)\sim d \ n^{-\tfrac 12} .
\end{align} 
Following the proof of the  upper bound, using  (\ref{help23}) and that $\mathbb{P}(Z_n>0)\leq e^{L_n}$ a.s., we have
\begin{align*}
 \mathbb{P}_1(Z_n=2) &= \gamma^n \mathbf{E}[\mathbb{P}(Z_n>0|\mathcal{E})^2]\leq \gamma^n \mathbf{E}[e^{2L_n}].
\end{align*}
Our assumptions imply $\mathbf{E}[X^2]<\infty$ and thus, as a direct consequence of \cite{kersting053}[Lemma 2.1 and $\int_{-\infty}^0 e^{x}u(-x)dx<\infty$],
\[\mathbf{E}[e^{2L_n}] = O(n^{-1/2}).\]
This proves the upper bound. 
\end{proof}
The next lemma describes the probability of $\{MRCA_n=n\}$ in the intermediate regime.
\begin{lemma}\label{mrcaint2}
Let $\mathbb{E}\big[Xe^{-X}\big]=0$ and assume that $\mathbb{E}\big[f''(1)\big]<\infty$. Then 
\[0<\liminf_{n\rightarrow\infty} n\ \mathbb{P}_1(MRCA_{n}=n|Z_{n}=2) \leq \limsup_{n\rightarrow\infty} n\ \mathbb{P}_1(MRCA_{n}=n|Z_{n}=2)  <\infty , \]
i.e. $\mathbb{P}_1(MRCA_{n}=n|Z_{n}=2)$ is of the order $n^{-1}$.
\end{lemma}
\begin{proof}
First, the event $\{MRCA_n= n\}$ implies that there are at least two individuals in generation $1$ and that 
from this generation on, at least two subtrees survive until generation $n$. We use the branching property and a decomposition according to the two subtrees which survive and get that
\begin{align}
\mathbb{P}&_1(MRCA_n= n, Z_n=2|\mathcal{E})\nonumber\\
&=\sum_{k=2}^{\infty} {k\choose 2} 
\mathbb{P}_1(Z_1=k|\mathcal{E}) \mathbb{P}(Z_{n}=1|\mathcal{E}, Z_1=1)^2\mathbb{P}(Z_{n}=0|\mathcal{E}, Z_1=1)^{k-2} \nonumber\\
&\leq \sum_{k=2}^{\infty} \frac{k(k-1)}{2} \mathbb{P}_1(Z_1=k|\mathcal{E})
 \mathbb{P}(Z_{n}=1|\mathcal{E}, Z_1=1)^2\nonumber\\
&\leq  \mathbb{P}(Z_{n}=1|\mathcal{E}, Z_1=1)^2 f''_1(1)\ \text{a.s.}, \label{eqint1}
\end{align}
since $f_1''(1)=\sum_{k=2}^\infty k (k-1) \mathbb{P}_1(Z_1=k|\mathcal{E})$. 
Using (\ref{eq_p1}) and independence yields 
\begin{align}
\mathbb{P}&_1(MRCA_n= n, Z_n=2)\nonumber\\
&\leq  \mathbb{E}[f''(1)]\mathbb{E}[ e^{-2S_{n-1}} \mathbb{P}(Z_{n-1}>0|\mathcal{E})^4]\leq  \mathbb{E}[f''(1)]\mathbb{E}[e^{-2S_{n-1}+4L_{n-1}}].\label{eqint2}
\end{align}
Again, we change to the measure $\mathbf{P}$. Note that the assumptions of the lemma  $\mathbb{E}[f''(1)]<\infty$ ensures that  $\mathbf{Var}(X)<\infty$. Using also Corollary \ref{corexp}, s we get
\begin{align*}
\mathbb{P}&_1(MRCA_n= n, Z_n=2)\leq\mathbb{E}[f''(1)] \gamma^{n-1}\mathbf{E}[e^{-S_{n-1}+4L_{n-1}}]=O\big(\gamma^{n} n^{-3/2}\big).
\end{align*}
Inserting this and applying Lemma \ref{helpint} yields
\begin{align*}
\limsup_{n\rightarrow\infty} n\ \mathbb{P}&_1(MRCA_n= n| Z_n=2)<\infty.
\end{align*}
For the lower bound, we use similar arguments. First, 
\begin{align}
\mathbb{P}&_1(MRCA_n= n, Z_n=2|\mathcal{E})\geq \mathbb{P}_1(Z_1=2|\mathcal{E}) \mathbb{P}_1(Z_{n-1}=1|\mathcal{E}, Z_1=1)^2\quad \text{a.s.}\nonumber
\end{align}
Let $c>0$. Taking the expectation and  using (\ref{eq_p1})  yields
\begin{align}
&\mathbb{P}_1(MRCA_n= n, Z_n=2) \nonumber\\
&\ \geq \mathbb{P}_1(Z_1=2) \mathbb{E}[e^{-2S_{n-1}}\mathbb{P}_1(Z_{n-1}>0|\mathcal{E})^4]\nonumber\\
& \ \geq \mathbb{P}_1(Z_1=2)\gamma^{n-1}e^{-c}\mathbf{E}\big[\mathbb{P}_1(Z_{n-1}>0\big|\mathcal{E})^4|L_{n-1}\geq 0, S_{n-1}\leq c\big]\mathbf{P}(L_{n-1}\geq 0, S_{n-1}\leq c). \nonumber
\end{align}
Moreover, by Lemma \ref{le_critical} and Jensen's inequality, 
\[\liminf_{n\rightarrow\infty} \mathbf{E}[\mathbb{P}_1(Z_{n}>0|\mathcal{E})^4|L_n\geq 0, S_n\leq c]>0.\] 
Applying Lemma \ref{helpint} again, we get that
\[\liminf_{n\rightarrow\infty} n\ \mathbb{P}_1(MRCA_n= n| Z_n=2)>0.\]
\end{proof}

The next Lemma describes the probability that the MRCA is neither at the beginning nor at the end:
\begin{lemma}\label{mrcaint3}
Let $\mathbb{E}\big[Xe^{-X}\big]=0$ and $\mathbb{E}[f''(1)/(1-f(0))^2]<\infty$. Then for every $\delta \in (0,1)$ 
\[0<\liminf_{n\rightarrow\infty} n^{3/2}\ \mathbb{P}_1(MRCA_{n}=\lceil \delta n\rceil |Z_{n}=2) \leq \limsup_{n\rightarrow\infty} n^{3/2}\ \mathbb{P}_1(MRCA_{n}=\lceil \delta n\rceil|Z_{n}=2)  <\infty, \]
i.e. $\mathbb{P}_1(MRCA_{n}=\lceil \delta n\rceil |Z_{n}=2)$ is of the order $n^{-3/2}$.
\end{lemma}
\begin{proof}%
First, the event $\{MRCA_n=\lceil \delta n\rceil, Z_n=2\}$ implies that the two individuals in generation $n$ 
stem from one individual in generation $n-\lceil \delta n\rceil=\lfloor (1-\delta) n\rfloor$. If there are $k$ 
individuals in this generation, there are $k$ possibilities for the 
ancestor from which the two surviving individuals in generation $n$ stem from. All others have to become extinct. 
We use the branching property and a decomposition according to the two subtrees which survive to get that a.s.
\begin{align}
 \mathbb{P}&(MRCA_n=\lceil\delta n\rceil, Z_n=2|\mathcal{E}) \nonumber \\
&= \sum_{k=1}^\infty \mathbb{P}(Z_{\lfloor (1-\delta) n\rfloor}=k|\mathcal{E}) k \mathbb{P}(Z_n=0|\mathcal{E},Z_{\lfloor (1-\delta) n\rfloor}=k-1)\nonumber \\
&\quad \qquad \mathbb{P}(MRCA_{n}=\lceil\delta n\rceil, Z_n=2|\mathcal{E}, Z_{\lfloor (1-\delta) n\rfloor}=1)  \nonumber \\
&=  \sum_{k=1}^\infty \mathbb{P}(Z_{\lfloor (1-\delta) n\rfloor}=k|\mathcal{E}) k \mathbb{P}(Z_n=0|\mathcal{E},Z_{\lfloor (1-\delta) n\rfloor}=1)^{k-1} \nonumber \\
&\quad \qquad \mathbb{P}(MRCA_{n}=\lceil\delta n\rceil, Z_n=2|\mathcal{E}, Z_{\lfloor (1-\delta) n\rfloor}=1) \nonumber
\end{align}
Next, we set
\[s:=(s(\mathcal{E})=\mathbb{P}(Z_{n}=0|\mathcal{E}, Z_{\lfloor (1-\delta) n\rfloor+1}=1)\]
and note that
\[\mathbb{P}(Z_n=0|\mathcal{E},Z_{\lfloor (1-\delta) n\rfloor}=1)= f_{\lfloor (1-\delta) n\rfloor}(s) . \]
Thus we get that a.s.
\begin{align}
\mathbb{P}&(MRCA_n=\lceil\delta n\rceil, Z_n=2|\mathcal{E}) \nonumber \\
 &=  f'_{0,\lfloor (1-\delta) n\rfloor}(f_{\lfloor (1-\delta) n\rfloor}(s)) 
\mathbb{P}_1(MRCA_{n}=\lceil\delta n\rceil, Z_n=2|\mathcal{E}, Z_{\lfloor (1-\delta) n\rfloor}=1) \label{int2402}
\end{align}
Next, using (\ref{eqint1}) and (\ref{eq_p1}), we get that
\begin{align}
& \mathbb{P}_1(MRCA_{ n}=\lceil\delta n\rceil, Z_n=2|\mathcal{E}, Z_{\lfloor (1-\delta) n\rfloor}=1)   \\
&\qquad  \leq f''_{\lfloor (1-\delta) n\rfloor+1}(1)  
\mathbb{P}(Z_{n}=1|\mathcal{E}, Z_{\lfloor (1-\delta) n\rfloor+1}=1)^2 \nonumber\\
&\qquad = f''_{\lfloor (1-\delta) n\rfloor+1}(1)  e^{-2(S_{n}-S_{\lfloor (1-\delta) n\rfloor+1})} \mathbb{P}(Z_{n}>0|\mathcal{E}, Z_{\lfloor (1-\delta) n\rfloor+1}=1)^4\nonumber\\
&\qquad = f''_{\lfloor (1-\delta) n\rfloor+1}(1)  e^{-2(S_{n}-S_{\lfloor (1-\delta) n\rfloor+1})} (1-s)^4.  \label{int2403}
\end{align}
Combining  (\ref{int2402}) and (\ref{int2403}), we have a.s.
\begin{align*}
\mathbb{P}&(MRCA_n=\lceil\delta n\rceil, Z_n=2|\mathcal{E})=  f'_{0,\lfloor (1-\delta) n\rfloor}(f_{\lfloor (1-\delta) n\rfloor}(s)) 
f''_{\lfloor (1-\delta) n\rfloor+1}(1)  e^{-2(S_{n}-S_{\lfloor (1-\delta) n\rfloor+1})} (1-s)^4.
\end{align*}
Moreover, from (\ref{genf30}), 
\begin{align*}
f'_{0,\lfloor (1-\delta) n\rfloor}&(f_{\lfloor (1-\delta) n\rfloor}(s)) \big(1- f_{\lfloor (1-\delta) n\rfloor}(s)\big)^2 
\leq e^{-S_{\lfloor (1-\delta) n\rfloor}} \mathbb{P}(Z_{\lfloor (1-\delta) n\rfloor}>0|\mathcal{E})^2 \ \text{a.s.}
\end{align*}
So
\begin{align*}
\mathbb{P}&(MRCA_n=\lceil\delta n\rceil, Z_n=2|\mathcal{E})  \nonumber \\ 
&= e^{-S_{\lfloor (1-\delta) n\rfloor}} \mathbb{P}(Z_{\lfloor (1-\delta) n\rfloor}>0|\mathcal{E})^2 
 f''_{\lfloor (1-\delta) n\rfloor+1}(1)  e^{-2(S_{n}-S_{\lfloor (1-\delta) n\rfloor+1})} \frac{(1-s)^4}{(1- f_{\lfloor (1-\delta) n\rfloor}(s))^2}.
\end{align*}
As already used in \cite{BK09}[Proof of Lemma 1], we have  for a generating function $f$ of a random variable $R$
\[\frac{1-f(s)}{1-s} = \sum_{k=0}^\infty s^k \mathbb{P}(R>k),\]
which is obviously increasing in $s$. Thus we get 
\begin{align*}
\frac{1-s}{1- f_{\lfloor (1-\delta) n\rfloor}(s)}\leq \big(\mathbb{P}(Z_{\lfloor (1-\delta) n\rfloor}>0|Z_{\lfloor (1-\delta) n\rfloor-1}=1,\mathcal{E})\big)^{-1}
= \frac{1}{1-f_{\lfloor (1-\delta) n\rfloor}(0)}.
\end{align*}
Combining these identities and using the independence of the environments yields :
\begin{align*}
\mathbb{P}&(MRCA_n=\lceil\delta n\rceil, Z_n=2)\\
&\leq \mathbb{E}\big[ e^{-S_{\lfloor (1-\delta) n\rfloor}} \mathbb{P}(Z_{\lfloor (1-\delta) n\rfloor}>0|\mathcal{E})^2\big] \mathbb{E}\big[ f''(1)/(1-f(0))^2 \big]
 \mathbb{E}\big[e^{-2S_{\lceil \delta n\rceil-1}}\mathbb{P}_1(Z_{\lceil \delta n\rceil-1}>0|\mathcal{E})^2\big],
\end{align*}
where we recall that  by  assumption $\mathbb{E}\big[ f''(1)/(1-f(0))^2 \big]<\infty$.
 As we have proved before, for every $\delta\in (0,1)$,
\[\mathbb{E}\big[ e^{-S_{\lfloor (1-\delta) n\rfloor}} \mathbb{P}(Z_{\lfloor (1-\delta) n\rfloor}>0|\mathcal{E})^2\big] \leq  \gamma^{\lfloor (1-\delta) n\rfloor}\mathbf{E}\big[
 e^{2L_{\lfloor (1-\delta) n\rfloor}}\big] = \gamma^{\lfloor (1-\delta) n\rfloor} O(n^{-1/2}) \]
and 
\begin{align*}
\mathbb{E}\big[e^{-2S_{\lceil \delta n\rceil-1}}\mathbb{P}_1(Z_{\lceil \delta n\rceil-1}>0|\mathcal{E})^2\big]&\leq \gamma^{\lceil \delta n\rceil-1} 
\mathbf{E}\big[e^{-S_{\lceil \delta n\rceil-1}+2L_{\lceil \delta n\rceil-1}}\big] =\gamma^{\lceil \delta n\rceil-1} O(n^{-3/2}).
\end{align*}
Together with Lemma \ref{helpint}, this yields the expected upper bound:
$$\limsup_{n\rightarrow\infty} n^{3/2}\ \mathbb{P}_1(MRCA_{n}=\lceil \delta n\rceil|Z_{n}=2)  <\infty.$$

For the lower bound, we use 
\begin{align*}
\mathbb{P}&_1(MRCA_n= \lfloor \delta n\rfloor , Z_n=2|\mathcal{E})\\
&\geq \mathbb{P}_1(Z_{\lfloor (1-\delta)n\rfloor} =1|\mathcal{E}) \mathbb{P}_1(MRCA_{n}=\lceil \delta n\rceil, Z_n=2|\mathcal{E}, Z_{\lfloor (1-\delta)n\rfloor}=1)\ \text{a.s.}\nonumber
\end{align*}
Both terms are independent and from Lemma \ref{mrcaint2}, we get
\[\liminf_{n\rightarrow\infty}\gamma^{-\lfloor \delta n\rfloor+1} n^{3/2}  \mathbb{P}_1(MRCA_{n}=\lceil \delta n\rceil, Z_n=2| Z_{\lfloor (1-\delta)n\rfloor}=1)>0 . \]
From the previous lemmas, 
\[\liminf_{n\rightarrow\infty}\gamma^{-\lfloor (1-\delta)n\rfloor} n^{1/2} \mathbb{P}_1(Z_{\lfloor (1-\delta)n\rfloor} =1)>0  . \]
Thanks to Lemma \ref{helpint}, we obtain the expected lower bound. 
\end{proof}
For the next proof, we require the following auxiliary result.
\begin{lemma}\label{le_suph}
We assume that $\mathbb{E}[Xe^{-X}]>0$. Then for every $c>0$, 
\begin{align*}
 \lim_{n\rightarrow\infty} \tfrac 1n \log \inf_{z\geq cn^2} \mathbb{P}_{z}(Z_n=2)& =- \infty, & \qquad
\limsup_{n\rightarrow\infty} \tfrac 1n \log \mathbb{P}_1(1\leq Z_n\leq cn^2) =\log\mathbb{E}[e^{-X}]. 
\end{align*}
\end{lemma}
\begin{proof}  
First, we observe that on the event $\{Z_n=2\}$, at most two initial subtrees  survive until generation $n$. 
Using  the  the branching property  conditionally  on $\mathcal{E}$, we have a.s.
\begin{align*}
 \inf_{z\geq cn^2} \mathbb{P}_{z}(Z_n=2|\mathcal{E}) &\leq \sum_{k=\lfloor cn^2\rfloor}^\infty {k \choose 2} \mathbb{P}_1(Z_n=0|\mathcal{E})^{\lfloor cn^2\rfloor-2} \mathbb{P}_1(1\leq Z_n\leq 2|\mathcal{E})^2 \\
 &\quad +\sum_{k=\lfloor cn^2\rfloor}^\infty {k \choose 1} \mathbb{P}_1(Z_n=0|\mathcal{E})^{\lfloor cn^2\rfloor-1} \mathbb{P}_1(1\leq Z_n\leq 2|\mathcal{E}) .
\end{align*}
Again, we use the geometric form of LF distributions to see that  $\mathbb{P}(1\leq Z_n\leq 2|\mathcal{E})\leq 2 \mathbb{P}(Z_n=1|\mathcal{E})$ a.s. Next, we use 
 \begin{align*}
\sum_{k=n}^\infty k(k-1) \alpha^{k-2} \leq n^2 \frac{\alpha^{n-2}}{(1-\alpha)^3}, \ \sum_{k=n}^\infty k \ \alpha^{k-1} \leq n \frac{\alpha^{n-1}}{(1-\alpha)^2}
\end{align*}
to get 
\begin{align*}
 &\inf_{z\geq cn^2} \mathbb{P}_{z}(Z_n=2|\mathcal{E}) \\
&\ \ \leq n^2 \bigg[c^2
\frac{\mathbb{P}_1(Z_n=0|\mathcal{E})^{\lfloor cn^2\rfloor-2} \mathbb{P}_1(Z_n=1|\mathcal{E})^2}{(1-\mathbb{P}_1(Z_n=0|\mathcal{E}))^3}+c
\frac{\mathbb{P}_1(Z_n=0|\mathcal{E})^{\lfloor cn^2\rfloor-1} \mathbb{P}_1(Z_n=1|\mathcal{E})}{(1-\mathbb{P}_1(Z_n=0|\mathcal{E}))^2}\bigg]\\
&\ \ =n^2 \bigg[c^2
\frac{\mathbb{P}_1(Z_n=0|\mathcal{E})^{\lfloor cn^2\rfloor-2} \mathbb{P}_1(Z_n=1|\mathcal{E})^2}{\mathbb{P}_1(Z_n>0|\mathcal{E})^3}+c
\frac{\mathbb{P}_1(Z_n=0|\mathcal{E})^{\lfloor cn^2\rfloor-1} \mathbb{P}_1(Z_n=1|\mathcal{E})}{\mathbb{P}_1(Z_n>0|\mathcal{E})^2}\bigg] \\
& \ \ = n^2 \bigg[c^2
\mathbb{P}_1(Z_n=0|\mathcal{E})^{\lfloor cn^2\rfloor-2} e^{-2S_n}\mathbb{P}_1(Z_n>0|\mathcal{E})+c
\mathbb{P}_1(Z_n=0|\mathcal{E})^{\lfloor cn^2\rfloor-1} e^{-S_n}\bigg].
\end{align*}
by using (\ref{utile}). Finally, as $\mathbb{P}(Z_n>0|\mathcal{E})\leq e^{S_n}$ a.s., we get that a.s.
\begin{align*}
 \inf_{z\geq cn^2} \mathbb{P}_{z}(Z_n=2|\mathcal{E})&\leq (c^2+c)n^2 
\mathbb{P}_1(Z_n=0|\mathcal{E})^{\lfloor cn^2\rfloor-1} e^{-S_n} .
\end{align*}
Next, we use again  the change of  measure 
\begin{align*}
 \mathbf{P}(X\in dx)=\frac{e^{-x} \mathbb{P}(X\in dx)}{\mathbb{E}[e^{-X}]} .
\end{align*}
Then
\begin{align*}
 \inf_{z\geq cn^2} \mathbb{P}_{z}(Z_n=2)\leq (c^2+c)n^2 \mathbb{E}[e^{-X}]^n \mathbf{E}[\mathbb{P}(Z_\infty=0|\mathcal{E})^{\lfloor cn^2\rfloor-2}] .
\end{align*}
Using Jensen's inequality yields
\begin{align*}
\limsup_{n\rightarrow\infty}\tfrac 1n \log \inf_{z\geq cn^2} \mathbb{P}_{z}(Z_n=2)&\leq \log\mathbb{E}[e^{-X}]+ \limsup_{n\rightarrow\infty}n\mathbf{E}[\log\mathbb{P}(Z_\infty=0|\mathcal{E})] .
\end{align*}
Finally, note that $\mathbb{E}[Xe^{-X}]>0$ implies $\mathbf{E}[X]>0$. Thus $\mathbf{P}(\mathbb{P}(Z_\infty=0|\mathcal{E})<1)>0$ and 
therefore  $\mathbf{E}[\log\mathbb{P}(Z_\infty=0|\mathcal{E})]<0$. This yields the first result. \medskip\\
For the second claim, we use again the geometric form of the probabilities of of LF  distributions to get
\begin{align*}
\mathbb{P}_1(1\leq Z_n\leq cn^2|\mathcal{E})\leq \lceil cn^2\rceil \mathbb{P}_1(Z_n=1|\mathcal{E}) .
\end{align*}
Thus, taking expectations yields
\begin{align*}
\limsup_{n\rightarrow\infty} \tfrac 1n\log\mathbb{P}_1(1\leq Z_n\leq cn^2)&\leq \limsup_{n\rightarrow\infty} \tfrac 1n\log \Big(\lceil cn^2\rceil \mathbb{P}_1(Z_n=1)\Big) \\
&=\limsup_{n\rightarrow\infty}\tfrac 1n \log \mathbb{P}_1(Z_n=1)= \log\mathbb{E}[e^{-X}]\  .
\end{align*}
where the last result has been shown in the previous subsection.
\end{proof}

\begin{lemma}\label{le3} We assume that  $\mathbb{E}[Xe^{-X}]>0$. Then,  for every $\delta\in(0,1]$,
\[\limsup_{n\rightarrow\infty}\tfrac 1n \log\mathbb{P}_1(MRCA_n>\delta n|Z_n=2)<0. \]
\end{lemma}
\begin{proof} First, we recall that the event $\{MRCA_n=\lfloor \delta n\rfloor+1\}$ implies that there are at least two individuals in generation 
$n-\lfloor \delta n\rfloor$ and that 
from this generation on, at least two subtrees survive until generation $n$. 
As in the preceding lemmas, we use the branching property and a decomposition 
according to the two subtrees which survive and get that
\begin{align*}
\mathbb{P}&_1(MRCA_n=\lfloor \delta n\rfloor+1, Z_n=2)\nonumber\\
&\leq \sum_{k=2}^{n^2} {k\choose 2} 
\mathbb{P}_1(Z_{n-\lfloor\delta n\rfloor}=k) 
\mathbb{E}\big[\mathbb{P}_1(Z_{\lfloor \delta n\rfloor}=1|\mathcal{E})^2\mathbb{P}(Z_{\lfloor \delta n\rfloor}=0|\mathcal{E})^{k-2}\big] 
+ \inf_{z\geq n^2} \mathbb{P}_z(Z_{\lfloor \delta n\rfloor}=2)\nonumber\\
&\leq  A_n+B_n,  
\end{align*}
where
$$A_n:=n^4 \mathbb{P}_1(1\leq Z_{n-\lfloor\delta n\rfloor}\leq n^2)\mathbb{E}\big[\mathbb{P}_1(Z_{\lfloor \delta n\rfloor}=1|\mathcal{E})^2\big], \qquad  
B_n:=\inf_{z\geq n^2} \mathbb{P}_z(Z_{\lfloor \delta n\rfloor}=2).\nonumber\\
$$
Letting $n$ go to $\infty$  and applying Lemma \ref{le_suph} yields
\begin{align}
&\limsup_{n\rightarrow\infty} \tfrac 1n \log  \mathbb{P}_1(MRCA_n=\lfloor \delta n\rfloor+1, Z_n=2) \nonumber \\
& \qquad \leq \max\Big\{ \limsup_{n\rightarrow\infty}\tfrac 1n \log A_n,  
\limsup_{n\rightarrow\infty}\tfrac 1n\log B_n\Big\} \nonumber \\
&\qquad =\max\Big\{ \limsup_{n\rightarrow\infty}\tfrac 1n \log A_n,  -\infty\Big\}=\limsup_{n\rightarrow\infty}\tfrac 1n \log A_n.\label{eq_sup1}
\end{align}
Next, let us treat the term named $A_n$. By Lemma \ref{le_suph},
\begin{align}\label{eq_sup2}
 \limsup_{n\rightarrow\infty} \tfrac{1}{n} \log \mathbb{P}(1\leq Z_{n-\lfloor\delta n\rfloor} \leq  n^2) = (1-\delta)\log\mathbb{E}[e^{-X}] . 
\end{align}
Using (\ref{utile}) and $\mathbb{P}(Z_n>0|\mathcal{E})\leq e^{L_n}$ a.s. yields
\begin{align}
 \mathbb{E}[&\mathbb{P}_1(Z_{\lfloor \delta n\rfloor}=1|\mathcal{E})^2]\leq \mathbb{E}[e^{-2S_{\lfloor \delta n\rfloor}}\mathbb{P}_1(Z_{\lfloor \delta n\rfloor}>0|\mathcal{E})^4]\leq\mathbb{E}[e^{-2S_{\lfloor \delta n\rfloor}+4L_{\lfloor \delta n\rfloor}}] . \label{eq_sup3}
\end{align}
We recall the change of measure
\begin{align*}
 \mathbf{P}(X\in dx)=\frac{e^{-x} \mathbb{P}(X\in dx)}{\mathbb{E}[e^{-X}]}.
\end{align*}
Then $\mathbb{E}[Xe^{-X}]>0$ assures that $\mathbf{E}[X]>0$ and under $\mathbf{P}$, 
$S$ is a random walk with positive drift. From (\ref{eq_sup1}), (\ref{eq_sup2}) and (\ref{eq_sup3}) 
we get
\begin{align*}
\limsup_{n\rightarrow\infty}\tfrac 1n \log \mathbb{P}&_1(MRCA_n=\lfloor \delta n\rfloor+1, Z_n=2)\nonumber \\
&\leq(1-\delta)\log\mathbb{E}[e^{-X}]+\limsup_{n\rightarrow\infty} 
\tfrac 1n\log\mathbb{E}[e^{-2S_{\lfloor \delta n\rfloor}+4L_{\lfloor \delta n\rfloor}}]\\
&=(1-\delta)\log\mathbb{E}[e^{-X}]+\delta \log\mathbb{E}[e^{-X}] +\limsup_{n\rightarrow\infty} 
\tfrac 1n\log\mathbf{E}[e^{-S_{\lfloor \delta n\rfloor}+4L_{\lfloor \delta n\rfloor}}]\\
&=\log\mathbb{E}[e^{-X}]+\limsup_{n\rightarrow\infty} 
\tfrac 1n\log\mathbf{E}[e^{-S_{\lfloor \delta n\rfloor}+4L_{\lfloor \delta n\rfloor}}].
\end{align*}
As  $\mathbb{E}[Xe^{-X}]>0$, we know from the previous subsection  
that $\lim_{n\rightarrow\infty} \tfrac 1n \log \mathbb{P}_1(Z_n=2)=\log\mathbb{E}[e^{-X}]$. So the previous inequality yields
\begin{align*}
\limsup_{n\rightarrow\infty}\tfrac 1n \log \mathbb{P}&_1(MRCA_n=\lfloor \delta n\rfloor+1| Z_n=2)\leq\limsup_{n\rightarrow\infty} 
\tfrac 1n\log\mathbf{E}[e^{-S_{\lfloor \delta n\rfloor}+4L_{\lfloor \delta n\rfloor}}].
\end{align*}
Finally, we  prove that $\limsup_{n\rightarrow\infty}\tfrac 1n \log\mathbf{E}[e^{-S_{\lfloor \delta n\rfloor}+4L_{\lfloor \delta n\rfloor}}]<0$ to conclude. 
Decomposing the expectation with $0<c<\mathbf{E}[X]$ and using $4L_{\lfloor \delta n\rfloor}-S_{\lfloor \delta n\rfloor}\leq 0$ yields
\begin{align*}
 \mathbf{E}[&e^{-S_{\lfloor \delta n\rfloor}+4L_{\lfloor \delta n\rfloor}}]\leq e^{-c\lfloor \delta n\rfloor} + 
\mathbf{P}(4L_{\lfloor\delta n\rfloor}-S_{\lfloor\delta n\rfloor}>-c \lfloor \delta n\rfloor) \\
&\leq e^{-c\lfloor \delta n\rfloor} + \mathbf{P}(S_{\lfloor\delta n\rfloor}<c \lfloor \delta n\rfloor) .
\end{align*}
As $0<c<\mathbf{E}[X]$, by standard results of large deviation theory, the probability on the right-hand side is exponentially small if 
$\mathbf{E}[e^{-sX}]=\mathbb{E}[e^{(-1-s)X}]<\infty$ for some $s>0$. This yields that
\[\limsup_{n\rightarrow\infty}\tfrac 1n \log\mathbb{P}_1(MRCA_n=\lfloor \delta n\rfloor+1|Z_n=2)<0\]
and thus
\[\limsup_{n\rightarrow\infty}\tfrac 1n \log\mathbb{P}_1(MRCA_n>\delta n|Z_n=2)<0.\]
\end{proof}

\section{Examples with two environments :   dependence on the initial and final population.}\label{examples}
In this section, we focus on the importance of the initial population.
\paragraph{Example  1 : the limits of $\tfrac{1}{n}\log\mathbb{P}_1(Z_n=i)$ and  $\tfrac{1}{n}\log\mathbb{P}_1(Z_n=j)$ may be both finite, negative but different.}
 Assume that the environment consists of  two states
$q_1$ and $q_2$ such that 
$$r:=\mathbb{P}(Q_1=q_1)=1-\mathbb{P}(Q_1=q_2)>0; \quad  q_1(1)=1; \quad 
q_2(0) = p, \  q_2(2) =1-p,$$
with $ p\in (0,1)$. Then
$$
\tfrac{1}{n} \log \mathbb{P}_1(Z_n=1) =\log r, \qquad
\tfrac{1}{n} \log \mathbb{P}_1(Z_n=2) \geq \max\big\{\log r; \log[(1-r)2(1-p)p]\big\}.
$$
where the term $\log r$ comes from a population which stays equal to $1$ in the environment sequence  $(q_1,q_1,q_1, \cdots)$ 
and the last term comes from a population which stays equal to $2$ 
in the environment sequence $(q_2,q_2,q_2, \cdots)$. Thus if $r$ is chosen small enough (i.e. $r<\tfrac{2 (1-p)p}{1+2(1-p)p}$), 
$$\lim_{n\rightarrow\infty}\tfrac{1}{n} \log \mathbb{P}_1(Z_n=1) < \lim_{n\rightarrow\infty} \tfrac{1}{n} \log \mathbb{P}_1(Z_n=2).$$

\paragraph{Example 2 : the limits of $\tfrac 1n\log\mathbb{P}_k(Z_n=k)$ and $\tfrac 1n\log\mathbb{P}_1(Z_n=k)$ with $k>1$ may be both finite, negative but different.}
Actually, in the case without extinction, we have 
$$\lim_{n\rightarrow \infty} \tfrac 1n\log\mathbb{P}_i(Z_n=k) =i\log\P(Z_1=1),$$
as soon as $\P_i(\exists n : Z_n=k)>0$ and the result is immediate.
We give here an example with $k=2$ and possible extinction. We  first observe that such an example is not possible with one environment, i.e. 
in the Galton-Watson case. Then we introduce a simple example in the random environment case  and check that it is not 
in contradiction to Theorem 2.1, before considering the asymptotic behavior of the probabilities involved.\\

Indeed, for Galton-Watson processes with reproduction law $q$ such that $q(0)>0$ and $q(2)>0$, the fact 
that $q(1)>0$ already ensures that the limits of $\tfrac 1n\log\mathbb{P}_2(Z_n=2)$ and $\tfrac 1n\log\mathbb{P}_1(Z_n=2)$ are equal. 
In the case $q(1)=0$, we get that
$$\mathbb{P}_1(Z_{n+1}=2) =\mathbb{P}_1(Z_1\geq 2, Z_{n+1}=2)=\sum_{i\geq 2} q(i)\P_i(Z_n=2)$$
whereas killing one of the initial individuals starting from $2$ and letting the other survive yields :
$$ \mathbb{P}_2(Z_{n+1}=2)\geq 2q(0)\sum_{i\geq 2} q(i)\P_i(Z_n=2).$$
Thus $\mathbb{P}_2(Z_{n+1}=2) \geq 2q(0)\mathbb{P}_1(Z_{n}=2)$. A converse inequality is clear, so the limits of $\tfrac 1n\log\mathbb{P}_2(Z_n=2)$ and $\tfrac 1n\log\mathbb{P}_1(Z_n=2)$
have to be equal in the Galton-Watson case with possible extinction.

Thus, we consider two environments to provide an example that the initial population size is also of importance even if extinction is possible.
More precisely, let the environment consist  of the  two states $q_1$ and $q_2$ such that 
\begin{align*}
&r:=\mathbb{P}(Q_1=q_1)=1-\mathbb{P}(Q_1=q_2)>0, \nonumber \\
& q_1(1)=p\quad , \quad q_1(a)=1-p, \nonumber \\
\quad 
&q_2(0) = p \quad , \quad q_2(2) =p \quad, \quad q_2(a)=1-2p, 
\end{align*}
with $p\in (0,\tfrac 12)$  and $a>2$. Note $k=1 \not\in Cl(\mathcal{I})$, so this example doesn't contradict
Theorem \ref{theo_rho1} where it is assumed that the initial population size is in $Cl(\mathcal{I})$.

To prove that $\mathbb{P}_1(Z_n=2)\gg \mathbb{P}_2(Z_n=2)$, we first observe that 
\be
\label{probab}
\liminf_{n\rightarrow\infty}\tfrac 1n \log \mathbb{P}_1(Z_n=2)\geq \log(rp),
\ee
 which comes from a population 
staying equal
 to $1$ before the last generation in the environment sequence $(q_1, q_1,\ldots, q_1,q_2)$. \medskip\\
Next, let us estimate the extinction probability, given the environment. We first observe that any BPVE 
whose environments are either $q_1$ or $q_2$ is stochastically larger than the Galton-Watson process with reproduction law (and unique environment) $q_2$. As a consequence, 
\[\mathbb{P}_1(Z_n=0|\mathcal{E}) \leq \mathbb{P}_1(Z_n=0|Q_1=q_2\ldots,Q_n=q_2)\leq \mathbb{P}(Z_\infty=0|Q_1=q_2,Q_2=q_2,\ldots )=:s_{e} \quad \text{a.s.} \] 
It is well-known that $s_{e}$ is given as the first fix point  of the generating function $f_2$ of $q_2$: 
\[ s_{e}=f_2(s_e)=p+ps_{e}^2 + (1-2p) s_{e}^a . \]

Let us now estimate $s_{e}$. 
For  $s=2p$, we have $2p>f_2(2p)=p+4p^3+(1-2p)2^ap^a$ if
$a$ is large enough since $p<1/2$. Thus $s_{e}\leq 2p$ if only $a$ is large enough.

We get  
then for all  $i\geq 1$, $k\leq n$, 
\[\mathbb{P}(Z_n=0|\mathcal{E},  Z_k=i)\leq s_{e}^i \leq (2p)^{i}\quad \text{a.s.} \]
Using this estimate and the explicit law of $\P(Z_{k+1} = \cdot \ \vert \ Z_k=2, Q_k=q_1)$, we obtain a.s.
\begin{align*}
  \mathbb{P}_2&(Z_n=2|\mathcal{E},Q_k=q_1,Z_k=2)\\
&= p^2 \mathbb{P}(Z_n=2|\mathcal{E},Z_{k+1}=2)+2(1-p)p \mathbb{P}(Z_n=2|\mathcal{E},Z_{k+1}=1+a)\\
&\qquad \qquad+(1-p)^2\mathbb{P}(Z_n=2|\mathcal{E},Z_{k+1}=2a)\\
&= \mathbb{P}(Z_n=2|\mathcal{E},Z_{k+1}=2) \Big(p^2 +2(1-p)p  {a+1\choose 2}\mathbb{P}(Z_n=0|\mathcal{E},Z_{k+1}=1+a-2)\\
&\qquad +(1-p)^2 {2a\choose 2}\mathbb{P}(Z_n=0|\mathcal{E},Z_{k+1}=2a-2)\Big)\\
&\leq \mathbb{P}(Z_n=2|\mathcal{E},Z_{k+1}=2) \Big(p^2 +2(1-p)p {a+1\choose 2}\big(2p\big)^{a-1}+(1-p)^2 {2a\choose 2}\big(2p\big)^{2a-2}\Big)
\end{align*}
If $p$ is small enough (depending on $a>2$), we get that a.s.
\begin{align*}
 \mathbb{P}_2&(Z_n=2|\mathcal{E},Q_k=q_1,Z_k=2)\leq \mathbb{P}(Z_n=2|\mathcal{E},Z_{k+1}=2) 3 p^2 .
\end{align*}
Analogously, if the environment $q_2$ occurs in generation $k$, we get that a.s.
\begin{align*}
  \mathbb{P}_2&(Z_n=2|\mathcal{E},Q_k=q_2,Z_k=2)\\
&= 2 p^2 \mathbb{P}(Z_n=2|\mathcal{E},Z_{k+1}=2)+p^2 \mathbb{P}(Z_n=2|\mathcal{E},Z_{k+1}=4)+2p(1-2p)\mathbb{P}(Z_n=2|\mathcal{E},Z_{k+1}=a)\\
&\quad +2 p(1-p) \mathbb{P}(Z_n=2|\mathcal{E},Z_{k+1}=a+2)+(1-2p)^2 \mathbb{P}(Z_n=2|\mathcal{E},Z_{k+1}=2a)\\
&\leq \mathbb{P}(Z_n=2|\mathcal{E},Z_{k+1}=2) 3p^2 .
\end{align*}
Next, note that the population starting from $Z_0=2$ is either always $\geq 2$ or extinct. Thus in each generation, 
there are at least two individuals and we may apply the estimates above for the subtrees emerging in generation $k$.   
Finally we get that
\begin{align*}
\limsup_{n\rightarrow\infty} \tfrac 1n\log\mathbb{P}_2(Z_n=2) =\limsup_{n\rightarrow\infty} \tfrac 1n\log\mathbb{E}\big[\mathbb{P}_2(Z_n=2|\mathcal{E})\big] \leq \log (3p^2).
\end{align*}
We now choose $p$ small enough such that $3p^2<rp$ and recall (\ref{probab}) to get
\begin{align*}
\limsup_{n\rightarrow\infty} \tfrac 1n\log\mathbb{P}_2(Z_n=2) <\liminf_{n\rightarrow\infty} \tfrac 1n\log\mathbb{P}_1(Z_n=2).
\end{align*}
Finally, we note that this example shows that, as in the the case without extinction in \cite{bansaye08}, the initial population may be of importance for the asymptotic of the probability of staying small,
 but alive. \\

\textbf{Acknowledgement.}  The authors wish to thank the anonymous referee for several comments and corrections of the earlier version of this paper, which have significantly improved its quality.
This work partially was funded by project MANEGE `Mod\`eles
Al\'eatoires en \'Ecologie, G\'en\'etique et \'Evolution'
09-BLAN-0215 of ANR (French national research agency),  Chair Modelisation Mathematique et Biodiversite VEOLIA-Ecole Polytechnique-MNHN-F.X. and the professorial chair Jean Marjoulet.

\bibliographystyle{apalike}

\begin{thebibliography}{99}
\bibitem{abkv11}
V.I. Afanasyev and C. B\"oinghoff and G. Kersting and V.A. Vatutin. Conditional limit theorems for intermediately subcritical branching processes in random environment.
\newblock \textsl{accepted for publication in Ann. Inst. H. Poincar\'e Probab. Statist., http://arxiv.org/abs/1108.2127} (2011).
\bibitem{abkv10}
V.I. Afanasyev and C. B\"oinghoff and G. Kersting and V.A. Vatutin. Limit theorems for a weakly subcritical branching process in random environment.
\newblock {\em J. Theoret. Probab.} {\bf25}  (2012) 703--732.
\bibitem{kersting052} 
V.I. Afanasyev and J. Geiger and G. Kersting and V.A. Vatutin. Functional limit theorems for strongly subcritical branching processes in random environment. 
\newblock {\em Stochastic Process. Appl.} {\bf 115} (2005) 1658--1676.
\bibitem{kersting053} 
V.I. Afanasyev and J. Geiger and G. Kersting and V.A. Vatutin. Criticality for branching processes in random environment.  
\newblock {\em Ann. Probab.} {\bf 33} (2005) 645--673.
\bibitem{agresti} A. Agresti. On the extinction times of varying and random environment     
branching processes. \textsl{J. Appl. Probab. }\textbf{12} (1975) 39--46.
\bibitem{athreya} K. B. Athreya. Large deviation rates for branching processes. I . Single type case. \emph{Ann. Appl. Probab. } \textbf{4}  (1994) 779--790. 
\bibitem{athreya71}
K.B. Athreya and S. Karlin. On branching processes with random environments: {I}, {II}.
\newblock {\em Ann. Math. Stat.} {\bf 42} (1971) 1499--1520, 1843--1858.
\bibitem{AN}  K. B. Athreya and P. E. Ney. \emph{Branching processes}. Dover Publications Inc. Mineola, NY  (2004).

\bibitem{bansaye08} 
V. Bansaye and J. Berestycki. Large deviations for branching processes in random environment.  
\newblock {\em Markov Process. Related fields.} {\bf 15} (2009) 493--524.
\bibitem{BB10}
V. Bansaye and C. B\"oinghoff. Upper large deviations for branching processes in random environment with heavy tails.
\newblock {\em Electron. J. Probab.} \textbf{16} (2011) 1900--1933.
\bibitem{boe1}
C. B\"oinghoff. Branching processes in random environment. Goethe-University Frankfurt/Main, {\em PhD Thesis}, Frankfurt (2010).
\bibitem{BK09} 
C. B\"oinghoff and G. Kersting. Upper large deviations of branching processes in a random environment - Offspring distributions with geometrically bounded tails.
\newblock \textsl{Stochastic Process. Appl.} \textbf{120} (2010) 2064--2077.
\bibitem{car} F. Caravenna. A local limit Theorem for random walks
conditioned to stay positive. \emph{Probab. Theory Related Fields} {\bf133} (2005) 508--530.

\bibitem{Dek} F. M. Dekking. On the survival probability of  a branching process in a finite state iid environment. \emph{Stochastic Process.  Appl.} {\bf 27} (1998) 151--157.
\bibitem{dembo}
A. Dembo and O. Zeitoni. Large Deviations Techniques and Applications.
\newblock Jones and Barlett Publishers International. London (1993).
\bibitem{reduced} K. Fleischmann,  V. Vatutin. Reduced subcritical Galton-Watson processes in a random environment. \emph{Adv. in Appl. Probab.} {\bf 31} (1999) 88--111.
\bibitem{FlWa}
K. Fleischmann and V. Wachtel. On the left tail asymptotics for the limit law of 
supercritical Galton-Watson processes in the Böttcher case.  \emph{Ann. Inst. Henri Poincaré Probab. Stat.}
{\bf 45}  (2009) 201--225. 
\bibitem{geiger99}
J. Geiger. Elementary new proofs of classical limit theorems for Galton-Watson processes. 
\newblock {\em J. Appl. Probab.} {\bf 36} (1999) 301--309.
\bibitem{kersting00} 
J. Geiger and G. Kersting. The survival probability of a critical branching process in a random environment.  
\newblock {\em Theory Probab. Appl.} {\bf 45} (2000) 517--525.
\bibitem{GKV03} 
J. Geiger and G. Kersting and V.A. Vatutin. Limit theorems for subcritical branching processes in random environment.
\newblock {\em Ann. Inst. Henri Poincaré Probab. Stat.} {\bf 39} (2003) 593--620. 
\bibitem{GL} Y. Guivarc'h and Q. Liu. Asymptotic properties of branching processes in random environment. \emph{C.R. Acad. Sci. Paris,} {\bf 332} (2001) 339--344.
\bibitem{Hambly} B. Hambly. On the limiting distribution of a supercritical branching process in random environment. \emph{J. Appl. Probab. } {\bf 29} (1992) 499--518.
\bibitem{HuangLiu} C. Huang and Q. Liu.  
Moments, moderate and large deviations for a branching process in a random environment. \emph{Stochastic Process. Appl.} {\bf 122} (2010) 522--545.
\bibitem{HuangLiu2}  C. Huang and Q. Liu. Convergence in $L^p$ and its exponential rate for a branching process in a random environment. 
\emph{Avialable via http://arxiv.org/abs/1011.0533}  (2011).     
\bibitem{LPP}  R. Lyons and R. Pemantle and Y. Peres. Conceptual proofs of $L\log L$ criteria for
mean behavior of branching processes. \emph{Ann. Probab.} {\bf 23} (1995) 1125--1138.
\bibitem{MH} M. Hutzenthaler. Supercritical branching diffusions in random environment. \newblock {\em Electronic Commun. Probab.} \textbf{16} (2011) 781--791.    
\bibitem{kozlov76}
M. V. Kozlov. On the asymptotic behavior of the probability of non-extinction for critical branching processes in a random environment.
\newblock {\em Theory Probab. Appl.} {\bf 21} (1976) 791--804.
\bibitem{kozlov06} 
M. V. Kozlov. On large deviations of branching processes in a random environment: geometric distribution of descendants.
\newblock {\em Discrete Math. Appl.} {\bf 16} (2006) 155--174.
\bibitem{kozlov10} 
M. V. Kozlov. On large deviations of strictly subcritical branching processes in a random environment with geometric distribution of progeny.  \emph{Theory Probab. Appl.} {\bf 54} (2010)  424--446.
\bibitem{marchi}
E. Marchi. When is the product of two concave functions concave?
\emph{Int. J. Math. Game Theory Algebra} {\bf 19}  (2010) 165--172.
\bibitem{ne} J. Neveu. Erasing a branching tree. \textsl{Adv. Apl. Probab.} \textbf{suppl.} (1986) 101--108.
\bibitem{Rouault} A. Rouault. Large deviations and branching processes. \emph{Proceedings of the 9th International Summer School on Probability Theory and Mathematical Statistics} (Sozopol, 1997). Pliska Stud. Math. Bulgar. {\bf 13} (2000) 15--38.
\bibitem{smith69}
W. L. Smith and W.E. Wilkinson. On branching processes in random environments.
\newblock {\em Ann. Math. Stat.} {\bf 40} (1969) 814--824.
\bibitem{VV}
V.A. Vatutin and V. Wachtel. Local probabilities for random walks conditioned to stay positive.
\newblock {\em Prob. Theory. Relat. Fields} {\bf 143} (2009) 177--217.

\end{thebibliography}
\end{document}